\documentclass[]{interact}
\usepackage{latexsym}
\usepackage{anyfontsize} 
\usepackage{tikz} 
\usetikzlibrary{cd} 
\usepackage{color}
\usepackage{hyperref}
\usepackage{url}
\usepackage{breakurl}
\newcommand{\bburl}[1]{\textcolor{blue}{\url{#1}}}
\usepackage{amsmath, amsthm}
\usepackage{thmtools}
\usepackage{thm-restate}
\usepackage{mathtools}

\makeatletter
\newcommand{\monthyear}[1]{%
  \def\@monthyear{\uppercase{#1}}}
\newcommand{\volnumber}[1]{%
  \def\@volnumber{\uppercase{#1}}}
\makeatother

\newcommand{\updown}[2]{#1^{(#2)}}

\theoremstyle{plain}
\numberwithin{equation}{section} 
\newtheorem{thm}{Theorem}[section] 
\newtheorem{theorem}[thm]{Theorem}
\newtheorem{lemma}[thm]{Lemma}

\newtheorem{problem}[thm]{Problem}

\numberwithin{table}{section} 
\numberwithin{figure}{section}

\begin{document}

\monthyear{Month Year}
\volnumber{Volume, Number}
\setcounter{page}{1}

\title{Fibonacci numbers and the probability of polygon formation using random length sticks}

\author{
\name{Mark Brennan\textsuperscript{a}\thanks{Email address: marksbrennan58@gmail.com}, 
Noah Callow\textsuperscript{a}\thanks{Email address: noahcallow2@gmail.com} \and
Tian Cao Lin\textsuperscript{b}\thanks{Email Address: tiancaolin23@mails.ucas.ac.cn}}
\affil{\textsuperscript{a}Independent Researcher, UK \\
\textsuperscript{b}University of Chinese Academy of Sciences, China
}}

\maketitle

{\bf Article type}: research 
\bigskip

\begin{abstract}
	We present two complementary proofs that, if the lengths of $n$ sticks are sampled at random, then the probability that no $p+1$ sticks can form a $(p+1)$-sided polygon can be expressed as the product of the reciprocals of a series of terms involving the $p$-step Fibonacci numbers. The first proof uses matrix algebra to extend the method previously used by Sudbury et al. \cite{Sudbury2025} to derive expressions for the probabilities of not being able to form triangles and quadrilaterals. The second alternative proof uses a different approach based on expressions for the minimum and maximum lengths of each stick that are compatible with the constraint of not being able to form a $(p+1)$-sided polygon, and provides insights into the structure of the probability expressions and the underlying reason that they include the Fibonacci numbers. Furthermore, the approach is developed in a generalised way that can, in principle, be applied to sticks randomly sampled from any probability distribution.
\end{abstract}


\section{Introduction}
The problem of random sticks is one of the classic problems in geometric probability with a long history, dating back to the \textbf{"broken stick problem"} posed by Lemoine in 1873~\cite{Lemoine1873}: a stick of unit length is broken at two random points, and one asks for the probability that the three resulting segments can form a triangle. Since then, many generalisations of this basic formulation have been studied. D'Andrea and Gomez~\cite{DAngelis2006} extended the problem to the case of \(n\) pieces, proving that the probability that a stick broken randomly into \(n\) pieces can form an \(n\)-gon is \(1 - n/2^{n-1}\). More generally, Verreault~\cite{Verreault2022a, Verreault2022b} and Mukerjee~\cite{Mukerjee2024} systematically studied the \((p+1)\)-gon version: given a stick broken into \(n\) pieces, what is the probability that no, every or any random subset of \(p+1\) pieces can form a \((p+1)\)-gon? Remarkably, the probability formulas for no subset of \(p+1\) pieces forming a \((p+1)\)-gon are deeply connected with the \(p\)-step Fibonacci numbers.

In a different direction, Petersen and Tenner~\cite{Petersen2020} introduced a model that is fundamentally different in its assumptions --- the \textbf{"pick-up sticks model"}. In this model, the random lengths of \(n\) sticks are drawn independently from a uniform distribution on \([0,1]\), and there is no constraint that they sum to one. They proved that the probability that \(n\) such sticks cannot form an \(n\)-gon is \(\frac{1}{(n-1)!}\)~\cite{Petersen2020}.

Within this framework, Sudbury et al. \cite{Sudbury2025} show that the probability of not being able to form a triangle from any three of $n$ sticks with independent random lengths selected from a uniform distribution on $[0,1]$ is
\begin{equation}
    {PN}_3^n=\ \prod_{i=1}^{n}\frac{1}{F_i},
\label{eq:no_triangle}
\end{equation}
where $F_i$ are the (2-step) Fibonacci numbers defined by $F_1=1,$ $F_2=1,$ and $F_i=F_{i-1}+F_{i-2}$ for $i > 2$.

To prove this result they used the constraint that, after the sticks have been sorted in order of increasing length, the lengths of all consecutive triples must satisfy $l_{i-2}+l_{i-1} \leq l_{i}$, which is a constraint that the minimum length, $l_i^{\text{min}}$, of each stick ($i>2$) is the sum of the lengths of the next two shortest sticks. They then used the fact that the order statistics of a uniform $[0,1]$ sample has the same distribution as the normalised cumulative sums of exponential variables. On integrating over the $n$ exponential variables, the Fibonacci numbers in (\ref{eq:no_triangle}) emerged in the exponents of the integrated exponentials.

They then used the same approach to prove that the probability of not being able to form a quadrilateral from any four of $n$ sticks with independent uniformly distributed random lengths from $[0, 1]$ is
\begin{equation}
{PN}_4^n=\frac{1}{T_n-\ T_{n-2}}\prod_{i=1}^{n-1}\frac{1}{T_i},
    \label{eq:no_4-gon}
\end{equation}
where $T_i$ are the are the Tribonacci (3-step Fibonacci) numbers defined by $T_1=1$, $T_2=1$, $T_3=2$, and $T_i=\ T_{i-1}+T_{i-2}+T_{i-3}\ $  for $i>3$.

The algebraic manipulations to prove this result were more involved and also made use of Tribonacci number recursion relationships to reformulate the exponents of the exponential probability distributions.  

At the start of their paper, Sudbury at al. \cite{Sudbury2025} suggested that the occurrence of the factorials of the Fibonacci and Tribonacci numbers in the above probability expressions was “surprising”. However, they concluded their paper by postulating that their results suggested a deeper underlying structure in the probability that no subset of $n$ independently sampled random stick lengths can form a $(p+1)$-gon (polygon with $p+1$ sides). They invited the reader to extend their framework to find a closed form expression of ${PN}_{p+1}^n$ for higher values of $p+1$, and also encouraged the reader to search for an alternative proof of the same Fibonorial law that would shed further light on the probabilistic structure they had uncovered.
In this paper, we respond to both of these challenges by first extending the method of Sudbury et al. \cite{Sudbury2025} by using matrix algebra to prove  Theorem \ref{theorem: PN_{p+1}^n}. 

\begin{restatable}{thm}{MainIdentity}
    \label{theorem: PN_{p+1}^n}
	If $n$ sticks have independent, random lengths chosen from a uniform distribution on the unit interval $[0,1]$, the probability that no $(p+1)$ of them can form a $(p+1)$-gon (a polygon with $p+1$ sides) is
	\begin{equation}
	    {PN}_{p+1}^n= \prod_{i=1}^{n-p+2}\frac{1}{F_i^{p}} \prod_{i=n-p+3}^{n}\frac{1}{F_i^{p}- \sum_{j=1}^{i-n+p-2}{jF_{i-j-1}^{p}}},
        \label{eq:main_result}
	\end{equation}
    where $\{F_i^{p}\}^{\infty}_{i=-\infty}$ is the sequence of $p$-step Fibonacci numbers generated using the recurrence relationship
    \begin{equation*}
        F_i^p=\sum_{j=1}^{p}F_{i-j}^p
    \end{equation*}
    and the initial conditions $F_1^p=1$ and $F_0^p=F_{-1}^p=\ldots=F_{-p+2}^p=0$.
\end{restatable}

We then go on to develop an alternative proof of this theorm, first for triangles and then for the general case of a $(p+1)$-gon. This alternative approach provides the deeper insight sought by Sudbury et al.~\cite{Sudbury2025} by revealing how the geometric constraints on the minimum and maximum allowable lengths of each stick, that arise from the requirement that no $p+1$ sticks can form a $(p+1)$-gon, are the fundamental reason for the appearance of the Fibonacci numbers in the probability expression.

The paper is organized as follows. Section \ref{MA_prelim} reviews the necessary tools from order statistics and exponential spacings required for the matrix algebra proof of Theorem \ref{theorem: PN_{p+1}^n}, which is then presented in Section \ref{MA_main_result}. Section \ref{geometric_approach} develops the alternative geometric proof based on extrema length constraints, starting in Section \ref{subsec:GA_prelim} with a preliminary proof for the simplest case of triangles. Section \ref{subsec:geom_main_result} then builds on this simplest case to prove the general case for a $(p+1)$-gon. The geometric proof is developed in a way that can, in principle, be applied to sticks drawn randomly from any bounded probability distribution. Section \ref{subsec:Unbounded} provides a further generalistaion of the method to unbounded probability distributions, and applies this to the particular case of an exponential distribution. Section \ref{literature} discusses the relationship between our results and the existing literature on both the broken stick and pick-up sticks problems. Section \ref{conclusion} offers concluding remarks and suggests directions for future research. Appendix 1 (Section \ref{Appendix}) presents the derivation of an algorithm for calculating the explicit form of linear functions that appear in the expressions for the maximum stick lengths compatible with the constraint of not being able to form a ($p+1$)-gon. Finally, Appendix 2 provides a third more concise, simple proof of our main result \ref{theorem: PN_{p+1}^n}, the development of which was informed by our geometric proof. The essence of this simple proof is consistent with a proof  reported by Kern \cite{Kern2025}, which was developed from a proof of the triangles case of the pick-up sticks problem produced by the large language model \textit{ChatGBT}. However, an important difference is that Kern's \cite{Kern2025} probability expression uses different $(k-1)-bonacci$ numbers (defined by using different initial terms) from the $p$-step Fibonacci numbers used here.

\section{Matrix Algebra Approach}
\subsection{Preliminaries} \label{MA_prelim}
We first denote
\begin{align} \label{matrix A_p}
     \mathbf{A}_p=\begin{pmatrix}
1& 0&\dots&0&p-1\\
1& 0&\dots&0&p-2\\
0&1&\dots&0&p-3\\
\vdots&\vdots&&\vdots&\vdots\\
0& 0&\dots&1&0
 \end{pmatrix}\in \mathbb{R}^{p\times p}.
 \end{align} 



Before proving Theorem \ref{theorem: PN_{p+1}^n}, we state a lemma that captures the algebraic relations generated by the integration process that will be used in devleoping the proof.

\begin{lemma}\label{lemma A}
Let matrix \(\mathbf{A}_p\), defined by (\ref{matrix A_p}), \(\mathbf{R}^1=(1,1,\dots,1)^T\in \mathbb{R}^p\), and \(\mathbf{R}^l=(R_1^l,R^l_2,\dots,R_p^l)^T\in\mathbb{R}^p\) satisfy
    \[\mathbf{R}^l=\mathbf{A}^{l-1}_p\mathbf{R}^{1},\ l=1,2,3,\dots\]
Then \[R_p^l=F_l^p,R_{p-1}^l=F_{l+1}^p\] and 
\[\quad R_{i}^l=F_{l+p-i}^p-\sum_{j=1}^{p-i-1}(p-i-j)F^p_{l+j-1},i=1,2,\dots,p-2,\]
 where \(F^p_l\) \((l=1,2,3,\dots,)\) is the \(p\)-step Fibonacci numbers.
\end{lemma}
\begin{proof}
\(\mathbf{R}^{l}=\mathbf{A}^{l-1}_p\mathbf{R}^1\) tell us that \(\mathbf{R}^{l}=\mathbf{A}_p\mathbf{R}^{l-1}\), that is to say

\begin{equation}\label{equations:1}
    \left\{
\begin{aligned}
  R_1^l&=R_1^{l-1}+(p-1)R_p^{l-1},\\
    R_2^l&=R_1^{l-1}+(p-2)R_p^{l-1},\\
    R_3^l&=R_2^{l-1}+(p-3)R_p^{l-1},\\
    \cdots\\
    R_p^l&=R_{p-1}^{l-1}.
\end{aligned}
\right.
\end{equation}
Therefore,
\begin{equation}\label{eq:preR^l_p}
\begin{aligned}
    R^{l+1}_p&=R^{l}_{p-1},\\&=R_{p-2}^{l-1}+R_p^{l-1},\\
    &=R^{l-2}_{p-3}+2R^{l-2}_p+R_p^{l-1},\\
    &\vdots\\
    &=R_{1}^{l-p+2}+\sum_{i=1}^{p-2} iR_p^{l-i}.
\end{aligned}
\end{equation}
From the first two formulas in equations (\ref{equations:1}),
\[R_2^l-R_1^l=-R_p^{l-1},\]
which means 
\[R_1^l=R_p^{l-1}+R_2^l.\]

Substituting the above expression into equation (\ref{eq:preR^l_p}) and successively applying the equations in system (\ref{equations:1}) from the second to the last, we obtain
\begin{align*}
    R_p^{l+1}&=R_1^{l-p+2}+\sum_{i=1}^{p-2}iR_p^{l-i}\\
    &=R_p^{l+1-p}+R_2^{l-p+2}+(p-2)R^{l-p+2}_p+\sum_{i=1}^{p-3} iR_p^{l-i}\\
    &=R_p^{l+1-p}+R_p^{l-p+2}+R_3^{l-p+3}+\sum_{i=1}^{p-3} iR_p^{l-i}\\
    &\vdots\\
    &=\sum_{i=1}^{p} R_p^{l+1-i} .
\end{align*}
Let \(\mathbf{R}^{-p+2}=(1,0,0,\dots,0)^T\in \mathbb{R}^p\), then \[\mathbf{R}^{t+2-p}=\mathbf{A}^t_p \mathbf{ R}^{2-p}=(1,1,\dots,1,0,0,\dots,0)^T\in \mathbb{R}^p, \text{ for } t<{p-1}.\]
The elements from the first to the (t+1)-th position of the vector are all one, and the remaining elements are zero, and 
\[\mathbf{R}^{1}=\mathbf{A}_p^{p-1}\mathbf{R}^{2-p}.\]
\(\mathbf{A}_p\) is invertible, which show that \(R_p^{-p+2}=\cdots=R_p^{0}=0\). Then this proves that \( R_p^l \) \((l=1,2,\dots)\) is precisely the \(p\)-step Fibonacci-type recurrence by definition. That's to say, \( R_p^l=F_p^l \) \((l=1,2,\dots)\).

Now we  use \(R_p^{m}\) \((m=1,2,\dots)\) to present \(R_i^{l}\). Come back to the relationship (\ref{equations:1}),

The last two equations imply that 
\[R_{p-1}^l=R_p^{l+1}=F^p_{l+1},\]
and 

\[R_{p-2}^l=R_{p-1}^{l+1}-R_{p}^l=R_p^{l+2}-R_p^l.\]
Substituting this into the third last equation (which represents \(R^l_{p-3}\)), we obtain
\[R_{p-3}^l=R^{l+1}_{p-2}-2R_p^l=R_p^{l+3}-R^{l+1}_p-2R_p^l,\]
Since
\[R_{i}^l=R_{i+1}^{l+1}-(p-i-1)R_{p}^l, \text{ for }i\leq p-1.\]
By induction, we obtain
\begin{align*}
    R_{i}^l&=R^{l+p-i}_p-\sum_{j=1}^{p-i-1}(p-i-j)R_p^{l+j-1}\\ &=F_{l+p-i}^p-\sum_{j=1}^{p-i-1}(p-i-j)F^p_{l+j-1}, \text{ for }i\leq p-2. \qedhere
\end{align*}
\end{proof}

\subsection{Main Result} \label{MA_main_result}
\begin{proof}[Proof of Theorem \ref{theorem: PN_{p+1}^n}]
Let $U_{(1)}\le U_{(2)}\le\dots\le U_{(n)}$ be the order statistics of $U_1,\dots,U_n$.  The condition that no \(p+1\) lengths can form a \((p+1)\)-gon is equivalent to
\begin{align}\label{first f}
\sum_{i=l}^{l+p-1} U_{(i)}\leq U_{(l+p)},\qquad \text{for all } 1\leq l\leq n-p.
\end{align}
It is well established (see,for example, Johnson et al. \cite{johnson1970continuous}) that the order statistics of a uniform sample can be represented via normalized sums of independent exponential variables, if \(x_1,\dots,x_{n+1}\stackrel{\mathrm{i.i.d.}}{\sim}\operatorname{Exp}(1)\) with density \(f(x)=e^{-x}\) and $S=x_1+\dots+x_{n+1}$, therefore,
\[
U_{(i)}\stackrel{d}{=}\frac{x_1+\dots+x_i}{S},\qquad i=1,\dots,n .
\]
Substituting this into (\ref{first f}) yields

\[
\begin{aligned}
\sum_{l=0}^{p-1}\frac{x_{1}+\cdots+x_{i+l}}{S}&\leq\frac{x_{1}+\cdots+x_{i+p}}{S}, \\
 (p-1)\sum_{ l=1}^ix_l+(p-2)&x_{i+1}+\dots +x_{i+p-2}\leq x_{i+p},
\end{aligned}
\]
where \(1\leq i\leq n-p\).  With \(x_{l}\) \((l=1,2,\dots,p)\) being unbounded between zero and infinity.  

The successive inequalities define a nested region in \(\mathbb{R}^{n}\), and the required probability corresponds to the volume of this region under the exponential density. To evaluate it, we integrate successively from \(x_{n}\) down to \(x_{1}\), each step producing a factor determined by the recurrence relation that will lead to the \(p\)-step Fibonacci pattern.

To make the subsequent formulas more concise, we introduce the notation \(L_j\), that
\begin{align} \label{simplnotations}
L_{j}= (p-1)\sum_{ l=1}^{j-p}x_l+(p-2)x_{j-p+1}+\dots +x_{j-2},
\end{align}
for \(j=p+1,p+2,\dots,n\). With \(L_j=0\) for \(j=1,2,\dots,p\).

The probability we need to calculate is
\begin{align*}
    PN^n_{p+1}=\int_{L_1}^\infty \int_{L_2}^\infty&\dots\int_{L_n}^\infty e^{-\sum_{i=1}^n x_i}dx_n\dots dx_2dx_1\\
    =\int_{L_1}^\infty\int_{L_2}^\infty&\cdots \int^\infty_{L_{n-1}} e^{-p\sum_{ l=1}^{n-p} x_l-(p-1)x_{n-p+1}-\dots -2x_{n-2}-x_{n-1}}dx_{n-1}\dots dx_2dx_1.
\end{align*}
For convenience, we introduce the vector \( \mathbf{R}^1,\mathbf{R}^2\in \mathbb{R}^p \) to rewrite this integral, 
\begin{align*}
    PN^n_{p+1}=\int_{L_1}^\infty \int_{L_2}^\infty&\dots\int_{L_n}^\infty e^{-R_1^1\sum_{l=1}^{n-p+1} x_l-R_2^1x_{n-p+2}-\dots -R_p^1x_n}dx_n\dots dx_2dx_1\\
    =\int_{L_1}^\infty\int_{L_2}^\infty&\cdots \int^\infty_{L_{n-1}} e^{-R_1^2\sum_{ l=1}^{n-p} x_{l}-R_2^2x_{n-p+1}-\dots -R_{p-1}^2x_{n-2}-R_p^2x_{n-1}}dx_{n-1}\dots dx_2dx_1.
\end{align*}

Here, the last term \(x_{n-p}\) in the \(\sum\)-summation has been taken out separately to compensate for the missing integrated \(x_n\).
Let the integral obtained after \(j\) integrations be denoted \(I_j^{p}\), for example,
\begin{align*}
    I_1^{p}=&\int_{L_n}^\infty e^{-R_1^1(\sum_{t=1}^{n-p} x_t+x_{n-p+1})-R_2^1x_{n-k+2}-\dots -R_p^1x_n}dx_n
    \\=&\frac{1}{R_p^1}e^{-R_1^2\sum_{ t=1}^{n-p} x_{t}-R_2^2x_{n-p+1}-\dots -R_{p-1}^2x_{n-2}-R_p^2x_{n-1}},
    \end{align*}
where \(L_n\) is given by (\ref{simplnotations}) and 
\[
\begin{aligned}
    R_1^2&=R_1^1+(p-1)R_p^1,\\
    R^2_i&=R_{i-1}^1+(p-i)R_p^1,\quad i=2,\dots,p.
\end{aligned}
\]
Now extend the relationship between \( \mathbf{R}^1 \) and \( \mathbf{R} ^2\) to the general case, i.e., the properties that \(\mathbf{R}^l\) possesses after \(l\) integrations. By induction, 
\begin{align*}
    I_l^{p}&=\prod_{t=1}^{l-1}\frac{1}{R^t_p}\int_{L_{n-l+1}}^\infty e^{-R_1^{l}(\sum_{t=1}^{n-l-p+1} x_t+x_{n-l-p+2})-R_2^{l}x_{n-l-p+3}-\dots -R_p^{l}x_{n-l+1}}dx_{n-l+1}
    \\&=\prod_{t=1}^l\frac{1}{R_p^t}e^{-R_1^{l+1}\sum_{ t=1}^{n-l-p+1} x_{t}-R_2^{l+1}x_{n-l-p+2}-\dots -R_{p-1}^{l+1}x_{n-l-1}-R_p^{l+1}x_{n-l}}.
\end{align*}
In the \(\Sigma\)-summation term of the second integrand, a factor is extracted outside, ensuring that the exponent always contains \(p\) terms.  Then 
\begin{align*}
R_1^{l+1}\sum_{ t=1}^{n-l-p+1} x_{t}&+R_2^{l+1}x_{n-l-p+2}+\dots +R_{p-1}^{l+1}x_{n-l-1}+R_p^{l+1}x_{n-l}\\&=R^l_p\times L_{n-l+1}+R_1^{l}(\sum_{t=1}^{n-l-p+1} x_t+x_{n-l-p+2})+R_2^{l}x_{n-l-p+3}+\dots +R_{p-1}^{l}x_{n-l}.
\end{align*}
By observing the relationship between the expression of \(L_{n-l+1}\) and the exponential term, we can obtain that \( \mathbf{R}^l \) satisfies the following recurrence relation, 
\[
\begin{aligned}
    R_1^{l+1}&=R_1^{l}+(p-1)R_p^{l},\\
    R_i^{l+1}&=R_{i-1}^{l}+(p-i)R_p^{l},\quad i=2,\dots,p.\\
\end{aligned}\]
 The transformation matrix \(\mathbf{A}_p\) we obtained here is identical to the matrix in (\ref{matrix A_p}).

Denote \(\mathbf{R}^l=(R_1^l,R_2^l,\dots,R_p^l)^T\), we have 
\[\mathbf{R}^{l+1}=\mathbf{A}_p\mathbf{R}^l=\mathbf{A}^l_p\mathbf{R}^1.\]
This relationship is the same as the condition stated in Lemma \ref{lemma A}. After \(n-p\) integrations, we are left with

\begin{align*}
    I_{n-p}^{p}=\prod _{l=1}^{n-p}\frac{1}{R_p^l}e^{{-R_1^{n-p+1} x_{1}-R_2^{n-p+1}x_{2}-\dots -R_{p-1}^{n-p+1}x_{p-1}-R_p^{n-p+1}x_{p}}}.
\end{align*}
Then, by integrating successively with respect to \( x_{n-p},x_{n-1},\dots, x_2 \), and \( x_1 \), we obtain
\[ PN^n_{p+1}=\prod_{l=1}^{n-p}\frac{1}{R^l_p} \prod_{i=1}^p\frac{1}{R^{n-p+1}_i}.\]
Note that in the second product, the terms involving \(R^{n-p+1}_p\) and \(R^{n-p+1}_{p-1} = R^{n-p+2}_p\) can be incorporated into the first product. Applying Lemma~\ref{lemma A}, together with the change of indices \(i' = n - i + 1\) and \(j' = p - i - j\) (renaming \(i', j'\) as \(i, j\)), completes the proof.
 
\end{proof}

\section{Geometric Approach} \label{geometric_approach}
\subsection{Preliminaries} \label{subsec:GA_prelim}

\begin{theorem}[Sudbury et al. \cite{Sudbury2025}]\label{theorem: NP3}
	If $n$ sticks have independent, random lengths chosen from a uniform distribution on the unit interval $[0,1]$, the probability that no three of them can form a triangle is
	\[
		\prod_{i=1}^{n}\frac{1}{F_i},
	\]
    where $F_i$ are the Fibonacci numbers defined by $F_1=1,$ $F_2=1,$ and $F_i=F_{i-1}+F_{i-2}$ for $i > 2$.
\end{theorem}

Before providing our alternative proof of Theorem \ref{theorem: NP3}, we state a lemma that is needed.

\begin{lemma}
    \label{lemma: integral_PN_p}
    Let $x_1, x_2, \dots, x_n$ be independent and identically distributed random variables defined on a probability space $(\Omega, \mathcal{F}, \mathbb{P})$ that represent the lengths of $n$ sticks. We assume that \(x_i\) follows an absolutely continuous distribution with a probability density function \(p(x)\) supported on the bounded interval \((0, M)\). If $l_1, \dots, l_n$ denote the order statistics of $x_1, \dots, x_n$, satisfying $l_1 \le l_2 \le \dots \le l_n$, then the probability that these lengths are compatible with a given requirement, R, is given by
    \begin{equation}
        PS=n!\int_{R} \int\ldots \int\int{p(l_n){dl}_np(l_{n-1}){dl}_{n-1}{\ldots p(l_2)dl}_{2}p(l_1)dl}_1,
        \label{eq:NP_p^n}
    \end{equation}
    where R represents the domain of integration (i.e. the ranges of stick lengths) compatible with the stated requirement.  
   \end{lemma}

\begin{proof}[Proof of Theorem \ref{theorem: NP3}]
	In order to satisfy the condition of not being able to form a triangle, our alternative geometric proof uses the same constraint as Sudbury et al. \cite{Sudbury2025}: after the sticks have been sorted in order of increasing length, the lengths of all consecutive triples must satisfy $l_{i-2}+l_{i-1} \leq l_{i}$. This constraint sets the limits on the minimum lengths of each stick compatible with the condition that it is not possible to form a triangle. Consequently, making the assignment $l_0=0$,
    \begin{equation}
    \begin{aligned}
    l_1^{\text{min}} &=0, \\
    l_i^{\text{min}} &=l_{i-1}+l_{i-2},  
    \label{eq: p=3,l_i_min}
    \end{aligned}
    \end{equation}
    for $i \geq 2$.
    
    In addition, we use the fact that the above limits on minimum stick lengths also place a constraint on the maximum length of each stick that is dependent on the length of the next shortest stick. If the $(i-1)$-th and $i$-th sticks have lengths $l_{i-1}$ and $l_i$, then the lengths of all subsequent sticks must satisfy the following conditions
    \[
    \begin{aligned}
    &l_{i+1} \ge {l_{i-1}+ l_i},\\
    &l_{i+2} \ge l_{i-1}+\ {2l}_i, \\
   & l_{i+3} \ge 2l_{i-1}+\ {3l}_i, \\
    \vdots \\
    &l_n \ge F_{n-i}l_{i-1}+\ {F_{n-i+1}l}_i.
    \end{aligned}
    \]
    However, for sticks drawn from a uniform distribution on $[0,1],$ $l_n\leq l_n^{\text{max}}$ and $l_n^{\text{max}}=1$ and therefore, for a given value of $l_{i-1}$, the maximum length of the $i$th stick ($i \leq n-1$) is given by
    \begin{equation}
        l_i^{\text{max}}=\frac{1-F_{n-i}l_{i-1}}{F_{n-i+1}}.
        \label{eq:max_stick3}
    \end{equation}
    For $i \leq (n-1)$, this equation gives the maximum length of each stick in terms of the length of its next shortest stick and the Fibonacci numbers. The expression reveals the reason why the Fibonacci numbers appear in (\ref{eq:no_triangle}) – it is because they arise in the determination of the maximum stick lengths compatible with the requirement for it not to be possible to form a triangle.

    The $l_i^{\text{min}}$ and $l_i^{\text{max}}$ values given by (\ref{eq: p=3,l_i_min}) (\ref{eq:max_stick3}) define the Lemma (\ref{lemma: integral_PN_p} domain of integration compatible with our requirement that no three sticks out of \(n\) can form a triangle. We now proceed to calculate the associated probability, $PN_3^n$, by evaluating the integrals in (\ref{eq:NP_p^n}) over this domain and using \(p(l)=1\) for a uniform distribution on \([0,1]\). 
    Evaluating the first (innermost) integral and then using $F_1=F_2=1$,
    \begin{multline*}
        \int_{l_{n-2}+l_{n-1}}^1 dl_n = 1-l_{n-2}-l_{n-1} = \frac{1}{F_1}(1-F_1 l_{n-2} - F_2 l_{n-1}) \\ =\frac{1}{F_1}(F_2 l_{n-1}^{\text{max}} - F_2 l_{n-1}).
    \end{multline*}
    The second integral then becomes
    \begin{multline*}
        \frac{1}{F_1}\int_{l_{n-3}+l_{n-2}}^{l_{n-1}^{\text{max}}}{\left(F_2l_{n-1}^{\text{max}}- F_2l_{n-1}\right){dl}_{n-1 }} = \frac{1}{{2F}_1F_2}\left(1- F_1l_{n-2}- F_2l_{n-2}- F_2l_{n-3}\right)^2 \\ =\frac{1}{{2F}_1F_2}\left(F_3l_{n-2}^{\text{max}}- F_3l_{n-2}\right)^2.
    \end{multline*}
    Continuing in this manner to progressively evaluate all but the final integral in (\ref{eq:NP_p^n}) gives,
    \[
        PN_3^n = \frac{n!}{\left(n-1\right)!F_1\ldots F_{n-1}}\int_{0}^{l_1^{\text{max}}}{\left( F_nl_1^{\text{max}}- F_nl_1\right)^{n-1}{dl}_{1}}.
    \]
    Evaluating the final integral using $l_1^{\text{max}}=1/F_n$, we find
    \begin{equation*}
    PN_3^n = \prod_{i=1}^{n}\frac{1}{F_i}. \qedhere
    \end{equation*}
\end{proof}



\begin{proof}[Proof of Lemma \ref{lemma: integral_PN_p}]
    For $n$ sticks with lengths drawn randomly from a probability density function $p(x)$, the probability that any one stick of the $n$ sticks has a length between $l_n$ and $l_n+{dl}_n$ and all other $n-1$ sticks are shorter than $l_n$ is
    \[
        np(l_n)dl_n\left\{P(x\le l_n)\right\}^{n-1}
    \]
    where \(P(x\le l_n) = \int_0^xp(x)dx.\)    
    In this case, the $n-1$ shortest sticks have lengths \(\le l_n\), drawn from a renormalised probability density function given by \({p(x)}\text{/}{P(\le l_n)}\), 
    and the probability that the second longest stick has a length between $l_{n-1}$ and $l_{n-1}+dl_{n-1}$ and all the other sticks are shorter than $l_{n-1}$ is
    \[
        (n-1)\left(\frac{p(l_{n-1})}{(P(x\le l_n)}\right)dl_{n-1}\left(\frac{(P(x\le l_{n-1})}{(P(x\le l_n)}\right)^{n-2}.
    \]
    Continuing this process, we see that the probability that the $n$ sticks have lengths between $l_{i}$ and $dl_i$, for $1 \leq i \leq n$, in order of increasing length is       
   \begin{equation}
         n!p(l_1){dl}_1\dots \ p(l_n){dl}_n.
        \label{eq: prob_li_dli}
    \end{equation}
    
    To find the probability that the lengths of the sticks satisfy some requirement, $R$, we simply integrate (\ref{eq: prob_li_dli}) over the ranges of stick lengths compatible that requirement to give
    \begin{equation*}
        PS=n!\int_{R} \int\ldots \int\int{p(l_n){dl}_np(l_{n-1}){dl}_{n-1}{\ldots p(l_2)dl}_{2}p(l_1)dl}_1,
            \end{equation*}
    where R represents the domain of integration (i.e. the ranges of stick lengths) compatible with the stated requirement. 
\end{proof}

\subsection{Main Result}
\label{subsec:geom_main_result}
Our alternative geometric proof of Theorem \ref{theorem: PN_{p+1}^n} is based on Lemma \ref{lemma: integral_PN_p} plus the following sequence of lemmas, each building on the previous one.

\begin{lemma}
    The minimum lengths for each stick compatible with the condition of not being able to form a $(p+1)$-gon are given by
    \begin{equation}
        l_i^{\text{min}} =
        \begin{dcases}
        0 & \text{for } i=1, \\
        l_{i-1} & \text{for } 1 < i < p+1, \\
        \sum_{j=1}^{p} l_{i-j} & \text{for } i \ge p+1.
        \end{dcases}
        \label{eq: p-gon_li_min}
    \end{equation}
    \label{lemma: l_min}
\end{lemma}

\begin{lemma}
    The maximum lengths for each stick compatible with the condition of not being able to form a $(p+1)$-gon are given by
    \begin{equation}
        l_i^{\text{max}}=
        \begin{dcases}
         \frac{1}{m_1}    &  \text{for } i=1, \\
    \frac{1}{m_i}-\frac{f_i\left(l_1,\ldots,l_{i-1}\right)}{m_i} & \text{for } 1<i\leq n-1, \\
    \frac{1}{m_n}  & \text{for } i=n,
        \end{dcases}
        \label{eq: p-gon_li_max}
    \end{equation}
    where the $f_i$’s are linear functions of the $l_i$’s with integer coefficients (which in some cases are equal to zero) and $m_1$ and the $m_i$'s are positive (non-zero) integer constants with $m_n$=1.
    \label{lemma: l_max}
\end{lemma}
\begin{lemma}
    The constants $m_i$ are given by 
    \[
    m_i=
    \begin{dcases}
    F_{n-i+1}^p-\sum_{j=1}^{p-i-1}{jF_{n-i-j}^p}  & \text{for } 1\le i\ \le p-2, \\
    F_{n-i+1}^p & \text{for } p-1 \leq i \leq n. \\
    \end{dcases}
    \]
    \label{lemma: mi_values}
\end{lemma}
\begin{lemma}
    For $i\geq2$,
    \[
    m_i\left(l_i^{\text{max}}-l_i^{\text{min}}\right)=m_{i-1}\left(l_{i-1}^{\text{max}}-l_i\right).
    \]
    \label{lemma: max_min}
\end{lemma}

\begin{lemma}
    The probability of not being able to form a $(p+1)$-gon from any $p+1$ out of $n$ sticks is given by
    \[
    {PN}_{p+1}^n=\prod_{i=1}^{n}{\frac{1}{m_i}.}
    \]
    \label{lemma: PN_mi}
\end{lemma}
\begin{proof}[Proof of Theorem \ref{theorem: PN_{p+1}^n}]
     Using Lemmas \ref{lemma: PN_mi} and \ref{lemma: mi_values} leads directly to
     \begin{equation*}
         {PN}_{p+1}^n=\prod_{i=p-1}^{n}\frac{1}{F_{n-i+1}^p}\prod_{i=1}^{p-2}\frac{1}{F_{n-i+1}^p-\sum_{j=1}^{p-i-1}{jF_{n-i-j}^p}}.
     \end{equation*}
    Finally, making the transformation $i \rightarrow n-i+1$ in both products gives
    \begin{equation*}
        {PN}_{p+1}^n=\prod_{i=1}^{n-p+2}\frac{1}{F_i^{p}}\prod_{i=n-p+3}^{n}\frac{1}{F_i^{p}-\sum_{j=1}^{i-n+p-2}{jF_{i-j-1}^{p}}}. \qedhere
    \end{equation*}
 \end{proof}

\begin{proof}[Proof of Lemma \ref{lemma: l_min}]
    Once the sticks have been sorted in order of increasing length, the requirement of not being able to form a $(p+1)$-gon results in the constraint that the lengths of all consecutive $(p+1)$-tuples must satisfy
     \[
     l_{i-p}+l_{i-p+1}+l_{i-p+2}+\ldots+l_{i-1} \leq l_{i}.
     \]
     Using this constraint and also the fact that the sticks are sorted in order of length, we identify the following minimum lengths for each stick compatible with the requirement of not being able to form a $(p+1)$-gon
     \begin{equation*}
        l_i^{\text{min}} =
        \begin{dcases}
        0 & \text{for } i=1, \\
        l_{i-1} & \text{for } 1 < i < p+1, \\
        \sum_{j=1}^{p} l_{i-j} & \text{for } i \ge p+1.
        \end{dcases}
        \qedhere
    \end{equation*}
\end{proof}

\begin{proof}[Proof of Lemma \ref{lemma: l_max}]
    Adopting the same approach as for our alternative proof of Theorem \ref{theorem: NP3}, if the $i$ shortest sticks have lengths $l_1,l_2,\ldots,l_i$, then using Lemma \ref{lemma: l_min} the lengths of the subsequent sticks must satisfy the following conditions     \begin{equation*}
    \begin{aligned}
        &l_{i+1} \ge f(l_1,\ldots,l_i), \\
        &l_{i+2} \ge f\left(l_1,\ldots,l_i,l_{i+1}\right)=f(l_1,\ldots,l_i), \\
        &\vdots \\
        &l_n \ge f\left(l_1,\ldots,l_i,l_{n-1}\right)=f\left(l_1,\ldots,l_i\right)=f_i\left(l_1,\ldots,l_{i-1}\right)+m_il_i,
    \end{aligned}
    \end{equation*}
    where the $f$'s and $f_i$ are linear functions of the $l_i$’s with integer coefficients (which in some cases are equal to zero) and $m_i$ is a non-zero integer constant.
    Now, $l_n \leq l_n^{\text{max}}=1$ for stick lengths drawn from a uniform probability distribution on ${[0,1]}$, therefore
    \begin{equation*}
        l_i^{\text{max}}=
        \begin{dcases}
         \frac{1}{m_1}    &  \text{for } i=1, \\
    \frac{1}{m_i}-\frac{f_i\left(l_1,\ldots,l_{i-1}\right)}{m_i} & \text{for } 1<i\leq n-1, \\
    \frac{1}{m_n}  & \text{for } i=n,
        \end{dcases}
    \end{equation*}
    where $m_n=1.$
\end{proof}
Note, an algorithm for determining the explicit form of the functions $f_i$ is derived in the appendix, though this is not required for our proof of Theorem \ref{theorem: PN_{p+1}^n}.
\begin{proof}[Proof of Lemma \ref{lemma: mi_values}]
    For stick lengths drawn from a uniform distribution on $[0,1],$ by definition
    \begin{equation}
        m_n=1=F_1^{p}.
        \label{eq: m_n}
    \end{equation}

    Using the $l_i^{\text{min}}$ values given by Lemma \ref{lemma: l_min} and the same approach that was used in the proof of Lemma \ref{lemma: l_max}, for $p+1 \leq i \leq n-1$, it is straightforward to see that
    \begin{equation}
        m_i=F_{n-i+1}^p.
        \label{eq: m_i}
    \end{equation}

    For $i<p+1$, we use the same approach to calculate $l_i^{\text{max}}=l_{p-q+1}^{\text{max}}$. With the sticks arranged in order of increasing length, we have
    \begin{flalign}
    \begin{aligned}
          &l_1 =	l_1,	\\
    &\vdots		\\
    &l_{p-q+1}	=	l_{p-q+1}, \\		
    &l_{p-q+2} \ge l_{p-q+1}, \\	
    &\vdots			\\
    &l_{p} \ge l_{p-q+1},		\\
    &l_{p+1} \ge l_1+l_2+\ldots+l_{p-q}+ql_{p-q+1},		\\
    &l_{p+2} \ge	f_{p+2}(l_1,l_2,\ldots l_{p-q})+2ql_{p-q+1},	\\	
    &\vdots			\\
    &l_n \ge	f_n(l_1,l_2,\ldots l_{p-q})+m_{p-q+1}l_{p-q+1}.
    \end{aligned}
    \label{eq: l_mins}
    \end{flalign}
    However, we would like to express $m_{p-q+1}$ in terms of the $p$-step Fibonacci numbers, $F_i^{p}$.

    To achieve this we need to transform the column of $q$ lots of $l_{p-q+1}$ above the row for $l_{p+1}$ in (\ref{eq: l_mins}) into $l_{p-q+1}$ multiplied by the sum of a number of $F_{i}^{p}$ sequences, each starting at $F_{1}^{p}$.  

    We first note that we can transform a sequence of $q$ 1's as follows
    \[
    \begin{pmatrix}
    1 \vphantom{F_1^{p}} \\
    1 \vphantom{F_1^{p}} \\
    1 \vphantom{F_1^{p}} \\
    1 \vphantom{F_1^{p}} \\
    \vdots \\
    1 \vphantom{F_1^{p}} \\
    1 \vphantom{F_1^{p}} 
    \end{pmatrix}
    \rightarrow
    \begin{pmatrix}
    F_1^{p} \\
    F_2^{p} \\
    F_3^{p} \\
    F_4^{p} \\
    \vdots \\
    F_{q-1}^{p} \\
    F_q^{p}
    \end{pmatrix}
    -F_2^{p}
    \begin{pmatrix}
    0 \vphantom{F_1^{p}} \\
    0 \vphantom{F_1^{p}} \\
    1 \vphantom{F_1^{p}} \\
    1 \vphantom{F_1^{p}} \\
    \vdots \\
    1 \vphantom{F_1^{p}} \\
    1 \vphantom{F_1^{p}}
    \end{pmatrix}
    -F_3^{p}
    \begin{pmatrix}
    0 \vphantom{F_1^{p}} \\
    0 \vphantom{F_1^{p}} \\
    0 \vphantom{F_1^{p}} \\
    1 \vphantom{F_1^{p}} \\
    \vdots \\
    1 \vphantom{F_1^{p}} \\
    1 \vphantom{F_1^{p}}
    \end{pmatrix}
    - \ldots - F_{q-1}^{p}
    \begin{pmatrix}
    0 \vphantom{F_1^{p}} \\
    0 \vphantom{F_1^{p}} \\
    0 \vphantom{F_1^{p}} \\
    0 \vphantom{F_1^{p}} \\
    \vdots \\
    0 \vphantom{F_1^{p}} \\
    1 \vphantom{F_1^{p}}
    \end{pmatrix}.
    \]
    Using ${S}_{j,q-j}(1)$ $\in \mathbb{R}^q$\ to represent a sequence of $j$ leading 0’s followed by $(q-j)$ 1’s,  and ${S}_{j,q-j}\left(F_i^{p}\right)$ to represent a sequence of $j$ leading 0’s followed by the first $(q-j) \ F_i^{p}$ numbers (starting at $F_1^{p}$), the above transformation can be written as
    \begin{equation}
        {S}_{0,q}(1)={S}_{0,q}(F_i^{p})-F_2^{p}{S}_{2,q-2}(1)-F_3^{p}{S}_{3,q-3}(1) - \ldots\ -  F_{q-1}^{p}{S}_{q-1,1}(1)
        \label{eq: C1_qq}
    \end{equation}
    and using the same approach, we also have
    \begin{equation}
        {S}_{1,q-1}(1)={S}_{1,q-1}(F_i^{p})-F_2^{p}{S}_{3,q-3}(1)-F_3^{p}{S}_{4,q-4}(1) - \ldots -  F_{q-2}^{p}{S}_{q-1,1}(1).
        \label{eq: C1_qq-1}
    \end{equation}
    We now note that we are only interested in transforming sequences containing a maximum number of $q \leq p$ numbers. Hence, for the $F_i^{p}$ numbers multiplying the ${S}_{j,k}(1)$ sequences in (\ref{eq: C1_qq}) and (\ref{eq: C1_qq-1}), $2 \leq i \leq p-1$ and, therefore, we have the recurrence relationship
    \begin{equation}
        F_{i+1}^{p} = 2F_i^{p}.
        \label{eq: rec_rel}
    \end{equation}
    
    Using (\ref{eq: rec_rel}) and subtracting (\ref{eq: C1_qq-1}) multiplied by two from (\ref{eq: C1_qq}) gives
    \[
    {S}_{0,q}(1)={S}_{0,q}(F_i^{p})-F_2^{p}{S}_{2,q-2}(1)+2({S}_{1,q-1}(1)-{S}_{1,q-1}(F_i^{p})),
    \]
    which using $F_2^{p}=1$, can be expressed as
    \begin{equation}
        {\big({S_{0,q}\left(1\right)}- {S_{1,q-1}\left(1\right)}\big)}=\big({S}_{1,q-1}(1)- {S}_{2,q-2}(1)\big)+S_{0,q}(F_i^{p})-2{S}_{1,q-1}(F_i^{p}).
        \label{eq: rec_rel16}
    \end{equation}
    Now, by definition
    \begin{equation}
        {S}_{q-1,1}(1)=S_{q-1,1}(F_i^{p})
        \label{eq: C1_q1}
    \end{equation}
    and
    \begin{equation}
        {{S}_{q-2,2}(1)=S}_{q-2,2}(F_i^{p})
        \label{eq: C1_q2}
    \end{equation}
    i.e.,
    \[
    {S}_{q-2,2}(1)-{S}_{q-1,1}(1)={S}_{q-2,2}(F_i^{p})-{S}_{q-1,1}(F_i^{p}).
    \]
    Using this starting value of ${S}_{q-2,2}(1)-{S}_{q-1,1}(1)$ and making the repeated application of the recurrence relationship given by (\ref{eq: rec_rel16}), for $q \geq 3$, we find
    \begin{equation}
        {S}_{0,q}(1)= {S}_{1,q-1}(1)+\sum_{j=1}^{q-1}{S}_{q-j,j}(F_i^{p}).
        \label{eq: simpler_rec_rel}
    \end{equation}
    Finally, starting with ${S}_{q-2,2}(1){=S}_{q-2,2}(F_i^{p})$ and making the repeated application of (\ref{eq: simpler_rec_rel}), for $q \geq 3$, we find
    \begin{equation}
        {S}_{0,q}(1)= {S}_{0,q}(F_i^{p})-\sum_{j=1}^{q-2}{jS}_{j+1,q-j-1}(F_i^{p}).
        \label{eq: C1_qq_express}
    \end{equation}
    For $q=1,2$ or $q>2$, we can now use (\ref{eq: C1_q1}), (\ref{eq: C1_q2}) and (\ref{eq: C1_qq_express}), respectively, to replace the sequence of $q$ lots of $l_{p-q+1}$ above the row for $l_{p+1}$ in (\ref{eq: l_mins}) by $l_{p-q+1}$ multiplied by sequences of ascending $F_i^{p}$ numbers.
    
    Continuing these sequences of ascending $F_i^{p}$ numbers down to the row for $l_n$ in (\ref{eq: l_mins}) gives
    \begin{equation}
        m_{p-q+1} =
        \begin{cases}
        F_{n-p+q}^{p}, & \text{for } q=1, 2, \\
        F_{n-p+q}^{p}-\sum_{j=1}^{q-2}{jF_{n-p+q-j-1}^{p}}, & \text{for } 3 \leq q \leq p.
        \end{cases}
        \label{eq: X_q q_12}
    \end{equation}

    Finally, making the transformation $p-q+1 \rightarrow i$ gives
    \begin{equation*}
        m_i=
        \begin{dcases}
        F_{n-i+1}^p-\sum_{j=1}^{p-i-1}{jF_{n-i-j}^p}  & \text{for } 1\le i\ \le p-2, \\
        F_{n-i+1}^p & \text{for } p-1 \leq i \leq n. \\
        \end{dcases}
        \qedhere
    \end{equation*}
\end{proof}
\begin{proof}[Proof of Lemma \ref{lemma: max_min}]
    We first note that (\ref{eq: p-gon_li_min}) and (\ref{eq: p-gon_li_max}) tell us that the $l_i^{\text{min}}$'s and $l_i^{\text{max}}$'s can each be expressed as a linear function of the $l_i$'s with integer coefficients. We then consider the following two sequences of stick lengths (arranged in order of increasing length).

    Let $l_1,l_2, \ldots, l_{i-1},l_i^{\text{max}},l_{i+1}^{\text{min}},l_{i+2}^{\text{min}},\ldots,l_{n-1}^{\text{min}},l_n^{\text{min}}$ be a sequence of stick lengths, which we will call Sequence 1. In this sequence, sticks $\left(i+1\right)$
    through to $n$ take their minimum permissible lengths and stick $i$ takes its maximum permissible length, which occurs when  $l_n^{\text{min}}=1$.
    
    Let $l_1,l_2,\ldots,l_{i-1}^{\text{max}},l_i^{\text{min}},l_{i+1}^{\text{min}},\ldots,l_{n-1}^{\text{min}}, l_n^{\text{min}}$ be a sequence of stick lengths, which we will call Sequence 2. In this sequence, sticks $i$ through to $n$ take their minimum permissible lengths, and stick $(i-1)$ takes on its maximum permissible length, which occurs when $l_n^{\text{min}}=1$.   
    
    Using the same approach as for the proof of Lemma \ref{lemma: l_max}, from Sequence 1 we find
    \begin{equation}
        1=m_il_i^{\text{max}}+m_{i-1}l_{i-1}+f_i(l_1,\ldots,l_{i-2}),
        \label{eq: seq1_mi}
    \end{equation}
    and from Sequence 2,
    \begin{equation}
        1=m_il_i^{\text{min}}+m_{i-1}l_{i-1}^{\text{max}}+f_i(l_1,\ldots,l_{i-2}),
        \label{eq: seq2_mi}
    \end{equation}
    where $f_i(l_1,\ldots,l_{i-2})$ is the same linear function of $l_1,\ldots,l_{i-2}$ in both (\ref{eq: seq1_mi}) and (\ref{eq: seq2_mi}).
    
    Equating the right-hand sides of (\ref{eq: seq1_mi}) and (\ref{eq: seq2_mi}) gives
    \begin{equation*}
        m_i{(l}_i^{\text{max}}-l_i^{\text{min}})=m_{i-1}{(l}_{i-1}^{\text{max}}-l_{i-1}) \text{  for  } i \leq n-1.
    \end{equation*}
    Also using the same approach as for the proof of Lemma \ref{lemma: l_max}
\[l_{n-1}^{max} = 1 -\sum_{j=2}^pl_{n-j}\]
which gives
\[l_{n-1}^{max} - l_{n-1} = 1 -\sum_{j=1}^pl_{n-j}= l_n^{max} - l_n^{min}\]
Form Lemma \ref{lemma: mi_values}, \(m_n=F_1^p=1\) and \(m_{n-1}=F_2^p=1\), and we therefore have
\begin{equation*}
        m_n{(l}_n^{\text{max}}-l_n^{\text{min}})=m_{n-1}{(l}_{n-1}^{\text{max}}-l_{n-1}) .
    \end{equation*}\qedhere
\end{proof}

It is noted that, in our alternative proof of Theorem \ref{theorem: NP3} (i.e. for $\left(p+1\right)$=3), the equivalence given by Lemma \ref{lemma: max_min} was established after performing each integral using the Fibonacci number recurrence relationships.

\begin{proof}[Proof of Lemma \ref{lemma: PN_mi}]
    Noting that, for stick lengths randomly drawn from a uniform distribution on $[0,1]$, $p(l)=1$, we use the minimum and maximum stick lengths given by Lemmas 3.4 and 3.6 to evaluate the integrals in the probability expression given by Lemma 3.2 subject to the condition of not being able to form a $(p+1)$-gon from any $p+1$ of $n$ sticks. The result of the innermost integral in (\ref{eq:NP_p^n}) contributes a factor to $PN_{p+1}^n$ which can be written as
    \begin{equation}
    \left(\frac{1}{(n-n+1)m_n}\right) m_n\left(l_n^{max}-\ l_n^{min}\right).
    \label{eq: PN_p^nInnermost}
    \end{equation}
    Using Lemma \ref{lemma: max_min}, (\ref{eq: PN_p^nInnermost}) can be rewritten as
    \[
    \left(\frac{1}{\left(n-n+1\right)m_n}\right)m_{n-1}\left(l_{n-1}^{\text{max}}-l_{n-1}\right).
    \]
    We now use this expression for the integral over $l_n$ as the integrand for the second integral in (\ref{eq:NP_p^n}), which becomes
    \[\begin{aligned}
    \frac{1}{m_n}\int_{l_{n-1}^{\text{min}}}^{l_{n-1}^{\text{max}}}{m_{n-1}}(l_{n-1}^{\text{max}}-l_{n-1}){dl}_{n-1}=\frac{1}{2m_{n}m_{n-1}}\Big(m_{n-1}\left(l_{n-1}^{\text{max}}-l_{n-1}^{\text{min}}\right)\Big)^2.
    \end{aligned}
    \]                     	
    Using Lemma \ref{lemma: max_min} again, this can be rewritten as
    \[
    \frac{1}{2m_{n}m_{n-1}}\Big(m_{n-2}\left(l_{n-2}^{\text{max}}-l_{n-2}\right)\Big)^2.
    \]
    We continue to apply the same process to progressively evaluate all but the final integral in (\ref{eq:NP_p^n}), to give
    \begin{equation}\label{integral:m_1}
    {PN}_{p+1}^n=\frac{n!}{\left(n-1\right)!m_{n}m_{n-1}m_{n-2}\ldots m_2}\int_{l_1^{\text{min}}=0}^{l_1^{\text{max}}}\left(m_{1}l_1^{\text{max}}-m_{1}l_1\right)^{n-1}{dl}_{1}
    \end{equation}
    which, using  $l_1^{\text{max}}=1/m_1$, integrates to
    \[
    {PN}_{p+1}^n=\prod_{i=1}^{n}\frac{1}{m_i}. \qedhere
    \]
\end{proof}

It is interesting to note that, by construction, the expression that results from performing the innermost integrals in equation (\ref{eq:NP_p^n}) over $l_n$ through to, say, $l_j$ represents the probability that the lengths of sticks $j$ through to $n$ all lie between their minimum and maximum values. After using Lemma \ref{lemma: max_min} to re-express the result of these innermost integrals, we find that the revised expression always goes to zero when $l_{j-1}$ is set to $l_{j-1}^{\text{max}}$. This behaviour arises because, when $l_{j-1}$ is set equal to $l_{j-1}^{\text{max}}$, this fixes the lengths of sticks $j$ to $n$ to a specific combination of exact lengths that culminate in $l_n$ being exactly equal to one (see proof of Lemma \ref{lemma: l_max}). Consequently, the ranges of integration for sticks $j$ to $n$ in (\ref{eq:NP_p^n}) are reduced to zero and, hence, the probability for this exact combination of lengths (for sticks $j$ through to $n$) goes to zero. The same behaviour was found in the evaluation of the integrals in our alternative proof of Theorem \ref{theorem: NP3} (i.e. for $p+1=3$), where we used the Fibonacci number recurrence relationships rather than Lemma \ref{lemma: max_min} to re-express the result of each integral.  

We also note that in the the case of a bounded probability distribution, taking the sample space as \([0, M]\) is essentially equivalent to taking it as \([0, 1]\). Specifically, for a sample \(x\) defined on \([0, M]\) with uniform probability density function \(p(x) = 1/M\), the scaling transformation \(y = x / M\) maps the sample space onto \([0, 1]\), and the corresponding probability density function becomes \(M p(M y)\). This transformation does not affect whether sticks drawn from the distribution can form a given polygon. Hence, for bounded probability distributions, without loss of generality, we may assume the upper bound to be \(1\). 

Furthermore, with the exception of Lemma \ref{lemma: PN_mi} (which gives the result of performing the integrals in the Lemma \ref{lemma: integral_PN_p}  probability expression), all of the other lemmas used in our proof (and the algorithm given in the appendix for determining the explicit form of the functions $f_i$ that appear in the Lemma \ref{lemma: l_max} expressions for $l_i^{\text{max}}$ are independent of the probability distribution that the sticks are assumed to be drawn from. Therefore, in principle, our geometric approach can be applied to the pick-up sticks problem where the stick lengths are randomly drawn from any bounded probability distribution on $[0,M]$, provided that a suitable iterative (or other) procedure can be identified for evaluating the integrals resulting from the application of Lemma \ref{lemma: integral_PN_p}. 

However, Lemma \ref{lemma: l_min} is affected by a change in the support for the probability density function from $[0,M]$, denoted \(\operatorname{supp}(p) = \{x \in \mathbb{R}^+ : p(x) > 0\}\), to \(\operatorname{supp}(p) = [aM, M]\), with \(0 < a < 1\): this change results in $l_1^{\text{min}}$ increasing from zero to $aM$. For the case of a uniform probability distribution on $[aM,M]$ with \(p(x) = 1/M(1-a)\) for \(x \in [aM,M]\), if \(aM \leq M/m_1\) (i.e. is less than or equal to the value of $l_1^{\text{max}}$ compatible with the requirement of not forming a $(p+1)$-gon), then we can obtain a modified value of $PN_{p+1}^n$ by adding a factor of \([1/M(1-a)]^n\) for the modified probability distribution, a scale factor of \(M^{n-1}\) for the lengths os sticks $2$ to $n$, and re-evaluating the integral in (\ref{integral:m_1}) with \(l_1^{\text{min}}\) equal to $aM$ and \(l_1^{\text{max}}\) equal to \(M/m_1\) to give
\begin{equation}
       \frac{(1-m_1a)^n}{(1-a)^n}\cdot\frac{1}{\prod_{i=1}^n m_i}, \text{ for }  a\leq 1/m_1 
\label {eq: trunctaed}
\end{equation} 

As we would expect, as \(l_1^{\text{min}} = aM\) approaches \(M/m_1 =l_1^{\text{max}}\) this modified probability of not being able to form a $(p+1)$-gon falls to zero. When $a$ exceeds $1/m_1$, the requirement for not being able to form a $(p+1)$-gon is necessarily violated and $PN_{p+1}^n =0$.  

We also note that equation (\ref{eq: trunctaed}) maintains invariance with respect to the scale factor M. However, if the uniform probability distribution is expressed in the form \(p(x) = 1/(M-b)\) for \(x \in [b,M]\), then equation (\ref{eq: trunctaed}) becomes
\begin{equation*}
       \frac{(M-m_1b)^n}{(M-b)^n}\cdot\frac{1}{\prod_{i=1}^n m_i}, \text{ for }  b\leq M/m_1 
\end{equation*} 
and the scale invariance is "hidden".

\subsection{Unbounded Probability Distributions}
\label{subsec:Unbounded}
We now consider the case where the probability density function, \(p(l)\), is unbounded. In this case, if we introduce a truncation bound \(M\), the probability that all stick lengths are at most \(M\) is \(\mathbb{P}\{l_i\leq M:i=1,2,\dots,n\}\). Then, the probability that they cannot form a (p+1)-gon is obtained as 
\[\begin{aligned}
    \mathbb{P}\{l_i\leq M:i=1,2,\dots,n\}\times n!\int_{l_1^{\text{min}}}^{l_1^{M,\text{max}}}\cdots\int_{l_n^{\text{min}}}^{l_n^{M,\text{max}}}\prod_{i=1}^n p(l_i)\mathrm{d}l_n\dots \mathrm{d}l_1, 
\end{aligned}\]
where the $l_i^{\text{min}}$ values are given by Lemma (\ref{lemma: l_min}) and \(l^{M,\text{max}}_i\) are obtained from Lemma (\ref{lemma: l_max}) by replacing the upper bound \(1\) with \(M\). 

Since \(p(x)\) is a probability density function, by the Lebesgue dominated convergence theorem, taking the limit as \(M\to\infty\) yields the desired probability
\begin{equation}
\begin{aligned}PN^n_{p+1}=n!\int_{l_1^{\text{min}}}^{+\infty}\cdots\int_{l_n^{\text{min}}}^{+\infty}\prod_{i=1}^n p(l_i)\mathrm{d}l_n\dots \mathrm{d}l_1,
\end{aligned}
\label{eq: Unbonded Int}
\end{equation}

Hence, it can be seen that Lemma (\ref{lemma: integral_PN_p}) remains valid for unbounded probability distributions. In this case, there is no constraint in the $l^{\text{max}}$ values, which we simply set equal to \({+\infty}\).  

\textbf{A particular example} is the exponential distribution with parameter \(1\), i.e., a probability density function given by \(p(x) = e^{-x}\) for \(x > 0\). In this case, we obtain 
\begin{equation}
\begin{aligned}PN^n_{p+1}=n!\int_{l_1^{\text{min}}}^{+\infty}\cdots\int_{l_n^{\text{min}}}^{+\infty}\prod_{i=1}^n e^{-l_i}\mathrm{d}l_n\dots \mathrm{d}l_1.
\end{aligned}
\label{eq: PN_p+1^nUnbonded Int}
\end{equation}
This integral was obtained by Mukerjee \cite{Mukerjee2024} when solving the broken stick problem. By successive integration, the explicit expression for arbitrary \(p\) and \(n\) shown below can be derived inductively
\begin{equation} 
\begin{aligned}PN^n_{p+1}=n!\prod_{k=1}^{n-p+2}\frac{1}{\updown{t}{p}_k}\prod_{k=n-p+3}^{n}\frac{1}{\updown{t}{p}_k-\sum_{j=1}^{p+k-n-2}j\updown{t}{p}_{k-j-1}},
\end{aligned}
\label{eq: PN_p+1^nUnbounded}
\end{equation}

where the sequence \(\{t_k^{(p)}\}\) satisfies \(t_1^{(p)} = 1\),  \(t_k^{(p)} = 0\) for \(k=0,-1,\dots,2-p\), and 
\[\updown{t}{p}_k=1+\sum_{j=1}^p\updown{t}{p}_{k-j}, \text{ k} \geq 2.\]
This expression is strikingly similar to the case of a bounded uniform distribution. If we define \(T_k=\updown{t}{p}_k-\updown{t}{p}_{k-1}\), it is straightforward to verify that \(T_k\) is governed by the recurrence relation \(T_k=\sum_{j=1}^pT_{k-j}\), subject to the initial conditions \(T_2=2\), \(T_1=1\), and \(T_{i-p}=0,i=3,\dots,p\). This identifies \(T_k\) as the \(p\) -step Fibonacci sequence.
Consequently, we know that \(\updown{t}{p}_k\) can also be expressed in terms of the \(p\)-step Fibonacci sequence as follows
\[
    t_k^{(p)}=\sum_{j=1}^kF_j^{p},
\]
and, with $t_k^{(p)}$ expressed in this form, the similarity with the case of the bounded uniform distribution is further enhanced. 

Given our earlier finding that the appearance of the Fibonacci numbers results from the form of the expressions for $l_i^{\text{max}}$ when stick lengths are randomly drawn from a bounded probability distribution, it is a little surprising that (\ref{eq: PN_p+1^nUnbounded}) still involves \(p\)-step Fibonacci numbers even though the $l_i^{\text{max}}$ values are infinite in this case.  It is also remarkable that the pick-up sticks probability of not being able to form a $(p+1)$-gon when the sticks are drawn from an unbounded exponential probability distribution is exactly the same as that for the broken stick problem (see equation (\ref{eq: p1_our_result})).

A special case of (\ref{eq: PN_p+1^nUnbounded}) is when \(p=2\), for which it is easy to verify that the Fibonacci sequence satisfies the following identity
\(\sum_{j=1}^kF_j=F_{k+2}-1\), which means \(\updown{t}{2}_k=F_{k+2}-1\). 
Then \[
\begin{aligned}PN^n_3=n!\prod_{k=1}^n\frac{1}{F_{k+2}-1}.
\end{aligned}\]
As for the case of a bounded uniform probability distribution, we note that (\ref{eq: PN_p+1^nUnbounded}) remains valid if a "length" scale factor, $a$, is introduced into the exponential probability distribution to give a normalised probability density function \(p(x) = ae^{-ax}\) for \(x > 0\). In this case, the two factors of $a$ appearing in the normalised probability density function cancel on evaluation of each integral in (\ref{eq: PN_p+1^nUnbonded Int}).

If we change the support for the exponential probability distribution to \(\operatorname{supp}(p) = [b, \infty]\), with \(0 < b < \infty\), then we have  \(p(x) = e^{-x}/e^b\) for \(b \le x \le \infty\) and \(l_1^{min}=b\). The modified probability distribution introduces a factor $exp(-nb)$ into equation (\ref{eq: PN_p+1^nUnbounded}). The change in $l_1^{min}$ only affects the lower limit of the outermost integral in equation (\ref{eq: PN_p+1^nUnbonded Int}), resulting in a second additional factor in equation (\ref{eq: PN_p+1^nUnbounded}) of \(exp[-(\updown{t}{p}_n-\sum_{j=1}^{p-2}j\updown{t}{p}_{n-j-1})]\). The overall result is that equation (\ref{eq: PN_p+1^nUnbounded}) is modified to
\begin{equation*} 
\begin{aligned}
PN^n_{p+1}=n!exp[-(nb+\updown{t}{p}_n-\sum_{j=1}^{p-2}j\updown{t}{p}_{n-j-1})]\prod_{k=1}^{n-p+2}\frac{1}{\updown{t}{p}_k}\prod_{k=n-p+3}^{n}\frac{1}{\updown{t}{p}_k-\sum_{j=1}^{p+k-n-2}j\updown{t}{p}_{k-j-1}}.
\end{aligned}
\end{equation*}

Lastly we note that, as for the case of bounded probability distributions, (\ref{eq: Unbonded Int}) is valid for any unbounded probability distribution provided that a suitable iterative (or other) procedure for evaluating the integrals in (\ref{eq: Unbonded Int}) can be identified.

\section{Relationship to Existing Literature} \label{literature}
We can reproduce the result of Petersen and Tenner \cite{Petersen2020} that the probability of $n$ sticks drawn randomly from $[0,1]$ cannot form an $n$-gon by setting $\left(p+1\right)=n$ in (\ref{eq:main_result}), which gives
\begin{equation}
    {PN}_n^n=\prod_{i=1}^{3}\frac{1}{F_i^{n-1}} \prod_{i=4}^{n}\frac{1}{F_i^{n-1}-\sum_{j=1}^{i-3}{jF_{i-j-1}^{n-1}}}.
    \label{eq:Petersen2020}
\end{equation}
Now $F_1^{n-1}=1$, $F_2^{n-1}=1$ and $F_3^{n-1}=2$. Also, summing the sequences of numbers in (\ref{eq: C1_qq_express}) gives 
\[
{q=F}_{q+1}^p-\sum_{j=1}^{q-2}jF_{q-j}^p,
\]
i.e.,
\begin{equation}
    i-1=F_{i}^{n-1}-\sum_{j=1}^{i-3}jF_{i-j-1}^{n-1}.
\end{equation}
Hence, (\ref{eq:Petersen2020}) reduces to
\[
{PN}_n^n=\frac{1}{\left(n-1\right)!}.
\]
For a stick of unit length broken at $\left(n-1\right)$ random positions to form $n$ random length pieces – the “broken stick problem" – the existing literature actually address three different problems (see, for example, Mukerjee \cite{Mukerjee2024}).

\begin{problem}
    What is the probability, ${PN}_{p+1}^n$, that no $p+1$ pieces out of $n$ can form a $\left(p+1\right)$-gon?
    \label{problem:p1}
\end{problem} 

\begin{problem}
    What is the probability, ${PA}_{p+1}^n$, that every combination of $p+1$ pieces out of $n$ can form a $\left(p+1\right)$-gon?
    \label{problem:p2}
\end{problem}

\begin{problem}
    If $p+1$ pieces are chosen at random out of the $n$ pieces, with all such choices being equally probable, what is the probability, ${PR}_{p+1}^n$, that the chosen pieces can form a $\left(p+1\right)$-gon?
    \label{problem:p3}
\end{problem}

\subsection{Broken Stick Problem \ref{problem:p1}}

For a stick of unit length broken at $\left(n-1\right)$ random positions to form $n$ random length pieces, Mukerjee \cite{Mukerjee2024} proves that no $p+1$ pieces out of $n$ can form a $\left(p+1\right)$-gon is given by
\begin{equation}
    {PN}_{p+1}^n=n!\prod_{r=1}^{n}\beta_r,
    \label{eq:problem1}
\end{equation}
where the $\beta_r$ factors are defined by the following algorithm. Let $e_1,\ldots,e_n$ in $\mathbb{R}^n$ be unit vectors, then define vectors $b_1,\ldots,b_n$ recursively as follows: $b_1=e_1, \ b_r=e_r+b_{r-1}$ for $r=2,\ldots , k$, and $b_r=e_r+\sum_{u=1}^{k}b_{r-u}$ for $r=k+1,\ldots, n$. Finally, ${(\beta_1,\ldots,\beta_{n)}}^T=\sum_{r=1}^{n}b_r$. 

With some minor adaptations, the geometric proof of Theorem \ref{theorem: PN_{p+1}^n} set out in Section \ref{subsec:geom_main_result} can also be applied to provide an alternative proof of (\ref{eq:problem1}) expressed as follows.

\begin{theorem}
    If a stick of unit length is broken at $\left(n-1\right)$ random positions to form $n$ random length pieces, the probability of not being able to form a $(p+1)$-gon from any combination of $\left(p+1\right)$ of these pieces is given by
    \[
    {PN}_{p+1}^n=n!\prod_{i=1}^{n-p+2}\frac{1}{{SF}_i^p} \prod_{i=n-p+3}^{n}\frac{1}{{SF}_i^p-\sum_{j=1}^{i-n+p-2}{j{SF}_{i-j-1}^p}}.
    \]
    where ${SF}_i^p = \sum_{j=1}^{i}F_i^p$, and $\left\{F_i^p\right\}_{i=-\infty}^\infty$ is the sequence of $p$-step Fibonacci numbers.
    \label{theorem: problem1}
\end{theorem}

Our alternative proof uses the equivalent sequence of lemmas to those used for the proof of Theorem \ref{theorem: PN_{p+1}^n} in Section \ref{subsec:geom_main_result} with appropriate modifications for the broken stick model. In the main, these differences arise because of the constraint on the length of the final stick piece once they have been arranged in order of increasing length, i.e.
\begin{equation}
    l_n=1-\sum_{j=1}^{n-1}{l_j}.
    \label{eq:p1_ln}
\end{equation}

\begin{lemma}
    If a stick of unit length is broken at $\left(n-1\right)$ random positions to form $n$ random length pieces, the probability of not being able to form a $(p+1)$-gon from any combination of $p+1$ of these pieces can be expressed as
    \begin{equation}
        {PN}_{p+1}^n=n!\left(n-1\right)!\int_{l_1^{\text{min}}}^{l_1^{\text{max}}}\ \int_{l_2^{\text{min}}}^{l_2^{\text{max}}} \ldots \int_{l_{n-1}^{\text{min}}}^{l_{n-1}^{\text{max}}}{dl_{n-1} \ldots dl_2 dl_1},
        \label{eq: p1_integral}
    \end{equation}
    where the  limits of integration represent the minimum and maximum lengths of the stick pieces in order of increasing length that are compatible with the requirement of not being able to for a $(p+1)$-gon.
    \label{lemma:p1_integral}
\end{lemma}

\begin{lemma}
    The minimum lengths of the broken stick pieces compatible with the requirement of not being able to for a (p+1)-gon are
    \begin{equation}
        l_i^{\text{min}} =
        \begin{dcases}
        0 & \text{for } i=1, \\
        l_{i-1} & \text{for } 1 < i < p, \\
        \sum_{j=1}^pl_{i-j} & \text{for } i \geq p.
        \end{dcases}
    \end{equation}
    \label{lemma: p1_li_min}
\end{lemma}

\begin{lemma}
    The maximum lengths of the broken stick pieces compatible with the requirement of not being able to form a $(p+1)$-gon are
    \begin{equation}
        l_i^{\text{max}}=
        \begin{dcases}
        \frac{1}{s_1}    &  \text{for } i=1, \\
        \frac{1}{s_i}-\frac{g_i\left(l_1,\ldots, l_{i-1}\right)}{s_i} & \text{for } 1<i\leq n-1, \\
        \frac{1}{s_n}  & \text{for } i=n,
        \end{dcases}
        \label{eq:p1_li_max}
    \end{equation}
    where the $g_i$’s are linear functions of the $l_i$’s with integer coefficients (which in some cases are equal to zero) and the $s_i$'s are positive (non-zero) integer constants with $s_n=1$.
     \label{lemma: p1_li_max}
   \end{lemma}
\begin{lemma}
    The integer constants $s_i$ take the following values
    \begin{equation}
        s_i=
        \begin{dcases}
        {SF}_{n-i+1}^p-\sum_{k=1}^{p-i-1}k{SF}_{n-i-k}^p    &  \text{for } 1 \leq i \leq p-2, \\
        {SF}_{n-i+1}^p   & \text{for } p-1 \leq i \leq n-1.
        \end{dcases}
        \label{eq:integer_consts}
    \end{equation}
    \label{lemma:p1_si_values}
\end{lemma}
\begin{lemma}
    For $i\geq2$,
    \begin{equation}
        s_i\left(l_i^{\text{max}}-l_i^{\text{min}}\right)=s_{i-1}\left(l_{i-1}^{\text{max}}-l_i\right).
    \end{equation}
    \label{lemma: p1_si_rec}
\end{lemma}

\begin{lemma}
    The probability of not being able to form a $(p+1)$-gon from any $(p+1)$ out of $n$ pieces is given by
    \begin{equation}
        {PN}_{p+1}^n=n!\prod_{i=1}^{n-1}{\frac{1}{s_i}}.
    \end{equation}
    \label{lemma:p1_PN_si}
\end{lemma}

\begin{proof}[Proof of Theorem \ref{theorem: problem1}]
Lemmas \ref{lemma:p1_PN_si} and \ref{lemma:p1_si_values} lead directly to
\[
{PN}_{p+1}^n=\ n!\prod_{i=p-1}^{n}\frac{1}{{SF}_{n-i+1}^p}\ \ \prod_{i=1}^{p-2}\frac{1}{{SF}_{n-i+1}^p-\ \sum_{j=1}^{p-i-1}{j{FS}_{n-i-j}^p}},
\]
where we have chosen to include an extra factor (equal to one) for $i=n$ in the first product for consistency with the equivalent expression for pick-up sticks case. Finally, making the transformation $i\rightarrow n-i+1$ in both products gives
\begin{equation}
    {PN}_{p+1}^n=\ n!\prod_{i=1}^{n-p+2}\frac{1}{{SF}_i^p}\ \ \prod_{i=n-p+3}^{n}\frac{1}{{SF}_i^p-\ \sum_{j=1}^{i-n+p-2}{j{SF}_{i-j-1}^p}}. \qedhere
    \label{eq: p1_our_result}
\end{equation}
\end{proof}

The equivalence of (\ref{eq: p1_our_result}) and (\ref{eq:problem1}) can be seen by comparing the derivation of the $s_{i}$ values in the proof of Lemma \ref{lemma:p1_si_values} below with Mukerjee's recursive algorithm \cite{Mukerjee2024}for the calculation of his $\beta_i$ factors. It is relatively straightforward to see that the two sets of factors are exactly equivalent, i.e. $s_{i}=\beta_i$.

Mukerjee’s paper \cite{Mukerjee2024} includes an appendix that demonstrates his expression for ${PN}_{p+1}^n$ is equivalent to that proved by Verreault \cite{Verreault2022b} which, although it involves the reciprocals of summated sequences of Fibonacci numbers and looks somewhat similar, is also different from (\ref{eq: p1_our_result}).

The form given in (\ref{eq: p1_our_result}) has the advantage of having a pleasing symmetry with the equivalent expression derived in Section \ref{subsec:geom_main_result} for the pick-up sticks case. This symmetry emphasises the close fundamental relationship between the two cases, with the only difference between the two probability expressions being an additional combinatorial factor of $n!$ in (\ref{eq: p1_our_result}), and the replacement of all $F_i^p$ terms by $\sum_{j=1}^{i}F_i^p$ as a result of the different constraints on the length of the longest stick/stick piece for the two cases. 

Also note that setting $\left(p+1\right)=n$ in (\ref{eq: p1_our_result}) gives
\begin{equation}
    {PN}_n^n=\ n!\prod_{i=1}^{3}\frac{1}{{SF}_i^{n-1}}\ \ \prod_{i=4}^{n}\frac{1}{{SF}_i^{n-1}-\ \sum_{j=1}^{i-3}{j{SF}_{i-j-1}^{n-1}}}.
    \label{eq: p1_p=n}
\end{equation}
Because the series of $F_i^{n-1}$ numbers (all starting at $F_1^{n-1}$) in (\ref{eq: p1_p=n}) go no higher than $F_n^{n-1}$, with the exception of ${SF}_1^{n-1}$ which is equal to one, for all other series we have ${SF}_i^{n-1}=2SF_i^{n-1}$ and, therefore, by comparison with (\ref{eq:Petersen2020}) we have
\begin{equation}
    {PN}_n^n=\ \frac{n!}{\left(n-1\right)!2^{n-1}}=\frac{n}{2^{n-1}},
\end{equation}
which is in agreement with the long established expression noted by Verreault \cite{Verreault2022a} and for which an alternative proof was provided by Andrea and Gomez \cite{DAngelis2006}.

\begin{proof}[Proof of Lemma \ref{lemma:p1_integral}]
    To begin, we note that there are $\left(n-1\right)!$ permutations of the order in which $\left(n-1\right)$ breaks can be introduced into the stick at the same $\left(n-1\right)$ positions along its length. In addition, there are $n!$ combinations of positions at which $\left(n-1\right)$ breaks can be introduced into the stick that all give the same combination of lengths for the resulting stick pieces (or, equivalently, there are $n!$ permutations of the resulting orders a given combination of the lengths of the resulting stick pieces can occur in). There are, therefore, a total combination of $n!(n-1)!$ ways in which the stick can be broken at $\left(n-1\right)$ random positions that all give the same combination of lengths for the resulting stick pieces.
    
    We also note that, when we arrange the resulting pieces in some defined order (in particular, in order of increasing length) and consider variations in the lengths of the individual pieces, then we only have $\left(n-1\right)$ degrees of freedom since, once lengths have been assigned to the first $\left(n-1\right)$ pieces, then the length of the $n$th piece is fixed by the constraint given in (\ref{eq:p1_ln}).

    Finally, we note that the random position of the breaks means that the lengths of the resulting pieces have a uniform random distribution.
    
    (\ref{eq: p1_integral}) follows directly from the above observations.
\end{proof}

\begin{proof}[Proof of Lemma \ref{lemma: p1_li_min}]
    This proof is identical to that from Lemma \ref{lemma: l_min} given in Section \ref{subsec:geom_main_result}.
\end{proof}

\begin{proof}[Proof of Lemma \ref{lemma: p1_li_max}]
    The proof is the same as that for Lemma \ref{lemma: l_max} in Section \ref{subsec:geom_main_result} except that in the last step we use (\ref{eq:p1_ln}) to set $l_n=1-\sum_{j=1}^{n-1}l_{j}$ rather than using the constraint $l_n \leq 1$. The resulting expressions for the $l_i^{\text{max}}$ values have exactly the same form, however we have used $g_i$ and $s_i$ rather than $f_i$ and $m_i$ to represent the linear functions and integer constants in (\ref{eq:p1_li_max}) to emphasise that these functions and constants are different from the ones in the pick-up sticks case. 
\end{proof}
\begin{proof}[Proof of Lemma \ref{lemma:p1_si_values}]
    We first consider the case where $p+1 \leq i \leq n-1$. Using the $l_i^{\text{min}}$ values given by Lemma \ref{lemma: p1_li_min} and using the same approach as for the proof of Lemma \ref{lemma: l_max}, we can calculate $l_i^{\text{max}}$ as follows. With the stick pieces arranged in order of increasing length, we have
    \begin{equation*}
    \begin{aligned}
    &l_1 =	l_1,	\\
    &\vdots		\\
    &l_{i-1}=l_{i-1}, \\		
    &l_i=F_1^pl_i , \\	
    &\vdots			\\
    &l_{i+1} \ge F_2^pl_i+g_{i+1}(l_{i-1},l_{i-2},\ldots, l_{i-p+1}),		\\
    &l_{i+2} \ge F_3^pl_i+g_{i+2}(l_{i-1},l_{i-2},\ldots, l_{i-p+1}),		\\
    &\vdots			\\
    &l_n \ge F_{n-i+1}^pl_i+g_n(l_{i-1},l_{i-2},\ldots, l_{i-p+1}).
    \end{aligned}
    \end{equation*}
    
    Applying the constraint $l_n \leq 1$ used in the proof of Lemma \ref{lemma: mi_values} gave  $m_i=F_{n-i+1}^p$ for $p+1\leq i\leq n-1$. However, as previously noted, for the broken sticks case we instead use (\ref{eq:p1_ln}) to set $l_n=1-\sum_{j=1}^{n-1}l_j$, which results in
    \begin{equation}
        s_i=\sum_{j=1}^{n-i+1}F_j^p,
    \label{eq: si1}
    \end{equation}
    for $p+1 \leq i \leq n-1$. 
    
    The determination of the $m_i$ values for $i \leq p+1$ in Section \ref{subsec:geom_main_result} involved the identification of a number of ascending sequences of $F_i^p$ numbers (each starting at $F_1^p$) extending down to the final, $l_n$, row in (\ref{eq: l_mins}). The final step was application of the constraint $l_n \leq 1$, which resulted in the $m_i$ values being identified as the sum of the final $F_i^p$ numbers in each of these sequences. As set out above for the calculation of the $s_i$ values for $p+1 \leq i \leq n-1$, the determination of the broken stick $s_i$ values for $i \leq p+1$ also follows the same process as in Section \ref{subsec:geom_main_result} except that at the last step use (\ref{eq:p1_ln}) to set $l_n=1-\sum_{j=1}^{n-1}l_j$. By comparison with the calculation above, it is clear that this results in the $s_i$ values being identified as the sum of the sums of the sequences of $F_i^p$ numbers identified in the Section \ref{subsec:geom_main_result}. calculation of the equivalent $m_i$ values, i.e.
    \begin{equation}
        s_{p-q+1}=
        \begin{dcases}
        \sum_{j=1}^{n-p+q}F_j^p    &  \text{for }  q=1,\ 2, \\
        \sum_{j=1}^{n-p+q}F_j^p -\sum_{k=1}^{q-2}k \sum_{j=1}^{n-p+q-k-1}F_j^p   & \text{for } 3 \leq q \leq p-1.
        \end{dcases}
        \label{eq: si2}
    \end{equation}
    Finally, making the transformation $p-q+1\rightarrow i$ and using ${SF}_i^p$ to represent $\sum_{j=1}^{i}F_i^p$, by combining (\ref{eq: si1}) and (\ref{eq: si2}) we have
    \begin{equation*}
        s_i=
        \begin{dcases}
        {SF}_{n-i+1}^p-\sum_{k=1}^{p-i-1}k{SF}_{n-i-k}^p    &  \text{for } 1 \leq i \leq p-2, \\
        {SF}_{n-i+1}^p   & \text{for } p-1 \leq i \leq n-1. \qedhere\
        \end{dcases}
    \end{equation*}
\end{proof}
\begin{proof}[Proof of Lemma \ref{lemma: p1_si_rec}]
    The proof is identical to that for Lemma \ref{lemma: max_min} in Section \ref{subsec:geom_main_result} but with $s_i$'s replacing $m_i$'s.
\end{proof}

\begin{proof}[Proof of Lemma \ref{lemma:p1_PN_si}]
    This lemma is essentially the same as Lemma \ref{lemma: PN_mi} in Section \ref{subsec:geom_main_result}; the only differences are that there are now only $\left(n-1\right)$ integrals that contribute to the probability expression and the $s_i$ are different integer constants from the $m_i$ that applied in the case of Lemma \ref{lemma: PN_mi}. In all other respects, the proof is the same as that for Lemma \ref{lemma: PN_mi}.
\end{proof}

\subsection{Pick-up Sticks Equivalent of Broken Stick Problem \ref{problem:p2}}
Note, in this Section 4.2 we will use $\Sigma^i$ as shorthand to denote $\sum_{j=1}^{i}l_j$ (where the sticks are arranged in order of non-decreasing length).

For smaller values of $p+1$, it is relatively straightforward to derive expressions for the probability that every choice of $p+1$ from $n$ random length sticks drawn from $[0,1]$ will form a $(p+1)$-gon by using Lemma 3.2 but noting that there are now several separate regions to the domain of integration that is compatible with the requirement that all choices of $p+1$ sticks will form a $(p+1)$-gon, ${PA}_{p+1}^n.$

We start by noting that the requirement for all choices of $p+1$ sticks forming a $(p+1)$-gon would be met if all $n$ sticks had equal lengths very close to zero. Hence, the only lower bound limits on stick lengths are those that result from the sticks being arranged in order of increasing length, i.e. 
 \begin{equation}
 \begin{aligned}
& l_1^{min}=0,\\
  &l_i^{min}=l_{i-1}   & \text{for } i \geq 2.\ \\
  \end{aligned}
   \label{eq: minimum_lengths}
 \end{equation}
 We then note that any set of stick length ranges compatible with the requirement for any choice of $p+1$ sticks forming a $(p+1)$-gon must satisfy one of two following conditions:

 either
 \begin{equation}
  \Sigma^p < 1 \text{ and }l_i^{\text{max}} \leq \Sigma^p\  \text{ for } p+1\leq i\leq n\ 
  \label{eq: condition_one}
  \end{equation}

  or
  \begin{equation}\Sigma^p\ > 1.\ 
 \label{eq: condition_two}
  \end{equation}
 In the case of the second of these conditions, the only constraint on the maximum lengths of sticks $p+1$ to $n$ is $l_i^{\text{max}}=1$.  

 The lower bound limits on minimum stick lengths given by  (\ref{eq: minimum_lengths}) mean that  (\ref{eq: condition_one}) can only be satisfied if the maximum lengths of sticks 1 to $p$ are all constrained as follows
  \begin{equation}
  l_i^{\text{max}A}=\frac{1 - \Sigma^{i-1}}{p+1-i}\  \text{ for } i\leq p.\
  \label{eq: maxA}
  \end{equation}
  If $l_1$  exceeds the maximum length given by (\ref{eq: maxA}), then the lower bound limits on the minimum lengths of sticks 2 to $p$ given by (\ref{eq: minimum_lengths}) mean that $\Sigma^p$\ will necessarily be $> 1$. If $l_1$ is less than the maximum length given by (\ref{eq: maxA}) but $l_2$ exceeds the maximum length given by (\ref{eq: maxA}), then the lower bound limits on the minimum lengths of sticks 3 to p given by (\ref{eq: minimum_lengths}) mean that $\Sigma^p$\ will necessarily be $> 1$, and so on up to and including the length of stick $p$ exceeding the maximum given by (\ref{eq: maxA}). 

The above considerations give rise to $p+1$ separate regions of the domain of integration compatible with the requirement that all choices of $p+1$ sticks will form a $(p+1)$-gon and, hence, we have 
\begin{equation}
\begin{aligned}
{PA}_{p+1}^n&=n!\int_{0}^{l_1^{\text{max}A}} \ldots\int_{l_{p-2}}^{l_{p-1}^{\text{max}A}}\int_{l_{p-1}}^{l_p^{\text{max}A}}\int_{l_p}^{\Sigma^p} \ldots\int_{l_{n-1}}^{\Sigma^p}{dl}_n \ldots  {dl}_{p+1} {dl}_p  {dl}_{p-1} \ldots {dl}_1 \\
&+n!\int_{l_1^{\text{maxA}}}^{1}\int_{l_1}^{1}  \ldots\int_{l_{n-1}}^{1}{{dl}_n{{\ldots dl}_{2}dl}_1}\\
&+n!\int_{0}^{l_1^{\text{max}A}}\int_{l_2^{\text{max}A}}^{1} \int_{l_2}^{l_1} \ldots\int_{l_{n-1}}^{1} {dl}_n  \ldots\ {dl}_3 {dl}_{2} {dl}_1 \\
&\vdots \\
&+n!\int_{0}^{l_1^{\text{max}A}}  \ldots\int_{l_{p-2}}^{l_{p-1}^{\text{max}A}}\int_{l_{p-1}}^{1} \ldots\int_{l_{n-1}}^{1}{dl}_n \ldots   {dl}_p  {dl}_{p-1} \ldots {dl}_1.
\end{aligned}
\label{eq: Full_PA_Integral}
    \end{equation}  
For the simplest case of $p+1=3$,  (\ref{eq: Full_PA_Integral}) reduces to 
\begin{equation}
\begin{aligned}
{PA}_{3}^n&=n!\int_{0}^{\frac{1}{2}} \int_{l_1}^{1-{l_1}}\int_{l_2}^{l_1+l_2} \ \ldots\int_{l_{n-1}}^{l_1+l_2}{dl}_n\ \ldots\  {dl}_3  {dl}_2\  {dl}_1\ \\
&+n!\int_{\frac{1}{2}}^{1}\int_{l_1}^{1} \ \ldots\int_{l_{n-1}}^{1}{{dl}_n{{\ldots dl}_{2}dl}_1}\\
&+n!\int_{0}^{\frac{1}{2}}\int_{1-l_1}^{1} \int_{l_2}^{1} \ \ldots\int_{l_{n-1}}^{1}\ {dl}_n\  \ldots\ {dl}_3\ {dl}_{2}\ {dl}_1.
\end{aligned}
\label{eq: Integral_PAA3N}
    \end{equation}  
It is relatively straightforward to evaluate all three integrals in (\ref{eq: Integral_PAA3N}) giving
\begin{equation}
    {PA}_3^n=\frac{1}{2^{n-2}}.
    \label{eq: PAA3N}
\end{equation}
Applying (\ref{eq: Full_PA_Integral}) in the same way for the case of $p+1=4$ results in

\begin{equation}
    {PA}_4^n=2 \left\{ {{\left(\frac{2}{3}\right)}^{n-3}-{\left(\frac{1}{2}\right)}^{n-2}}\right\}.
    \label{eq: PAA4N}
\end{equation}
The probability expressions given by (\ref{eq: PAA3N}) and (\ref{eq: PAA4N}) are much simpler than the relatively complex general expressions proved by Verreault \cite{Verreault2022b} and Mukerjee \cite{Mukerjee2024} (which are different but equivalent) for the probability that every choice of  $p+1$ pieces of a stick broken into $n$ parts will form a $(p+1)$-gon. Consequently, there is no obvious relationship between the probability expressions for the broken stick and pick-up sticks models. Moreover, (\ref{eq: Full_PA_Integral}) for the pick-up sticks case involves $p+1$ $n$-dimensional integral contributions to $PA_{p+1}^n$, and the evaluation of each of these $n$-dimensional integrals is progressively more complex with increasing values of $p+1$. Therefore, it is unclear whether (\ref{eq: Full_PA_Integral}) provides a basis for deriving a general expression for $PA_{p+1}^n$ in the case of the pick-up sticks model.

\subsection{Pick-up Sticks Equivalent of Broken Stick Problem \ref{problem:p3}}
Mukerjee \cite{Mukerjee2024} proves a theorem (Theorem 3) that the probability of $p+1$ pieces chosen at random out of the $n$ pieces of a broken stick, with all such choices being equally probable, do not form a $(p+1)$-gon is equal to $PN_{p+1}^{p+1}$, i.e. the probability of not being able to (p+1)-gon when a stick is broken into $p+1$ pieces. This proof is based on a demonstration that the distribution of piece lengths resulting a random selection of $p+1$ pieces from the $n$ random length pieces of broken stick results in the same probability of not forming a $(p+1)$-gon as the random piece length distribution from a stick broken into $p+1$ pieces. As such, it is not dependent on the fact that the random length pieces come from broken sticks, and is equally applicable to the pick-up sticks problem where the stick lengths are randomly selected from $[0,1]$. 

Applying Mukerjee’s Theorem 3 \cite{Mukerjee2024} to the pick-up sticks problem, we find that if $p+1$  sticks are chosen at random out of the $n$ sticks drawn from $[0,1]$, with all such choices equally probable, then the probability that the chosen pieces can form a $(p+1)$-gon is given by
\begin{equation}
       PR_{p+1}^n = (1 - PN_{p+1}^{p+1})  = 1 - \frac{1}{p!}.
        \label{eq:Mukerjee2024}
\end{equation}

\section{Closing Remarks} \label{conclusion}
As encouraged by Sudbury et al. \cite{Sudbury2025}, we have extended their method by using matrix algebra to derive a generalised expression for the probability of not being able to form a $(p+1)$-gon from any $p+1$ of $n$ sticks with independent random lengths selected from a uniform distribution on $[0, 1]$.

We have also developed an alternative proof for this probability expression based on the minimum and maximum lengths of each stick compatible with the requirement of not being able to from a $(p+1)$-gon. This alternative proof reveals that it is the geometrical constraints on the minimum and maximum stick lengths that are the underlying reason for the appearance of the Fibonacci numbers in the probability expression. 
The alternative geometric proof has been developed in such a way that it can, in principle, be applied to sticks drawn randomly from any (bounded or unbounded) probability distribution, provided that a suitable iterative (or other) procedure can be identified for evaluating the integrals resulting from the application of Lemma \ref{lemma: integral_PN_p}. 

In this paper, we have applied the approach to the particular cases of sticks randomly drawn from a bounded uniform probability distribution and an unbounded exponential distribution. Remarkably, the case of sticks drawn randomly from an unbounded exponential distribution results in exactly the same expression for the probability of not being a form a $(p+1)$-gon as for the broken stick problem. Identifying other stick length probability distributions where our geometric approach can be applied is suggested as an area for further research. In the Appendix, we provide an algorithm that can be used to generate the explicit form of the linear functions in the expressions for \(l_i^{\text{max}}\) that may be useful in any such future research.

We have also used our geometric approach to prove an alternative expression (to those previously derived by Verreault \cite{Verreault2022b} and Mukerjee \cite{Mukerjee2024}) for the probability of not being able to form a  $(p+1)$-gon from any $p+1$ of the $n$ pieces formed when a stick of unit length is broken at $(n-1)$ positions. This alternative expression for the 
broken sticks probability has the advantage of having a pleasing symmetry with the expression we have derived for the pick-up sticks probability, thus emphasising the close fundamental relationship between the two cases. 
Finally, for the pick-up sticks case, we have used our geometric approach to derive the probability of all combinations of $p+1$ of $n$ sticks forming a $(p+1)$-gon for the two lowest values of $p+1$, 3 for triangles and 4 for quadrilaterals. Finding a generalised expression for this probability for any value of $p+1$ is an outstanding problem that also merits further research.
\section*{Disclosure statement}
No conflict of interests was reported by the authors.

\bibliographystyle{unsrt}
\bibliography{references}   
\section{Appendix 1: Algorithm for $l^{\text{max}}$ Functions} \label{Appendix}
In this appendix, we derive an algorithm that can be used to generate the explicit form of the linear functions in the expressions for \(l_i^{\text{max}}\) in the case of sticks randomly drawn from a probability distribution with an upper bound of \(1\). For a fixed \(p\geq 2\). Firstly, for \(i>p-1\), we define \(\updown{e}{p}_{2-k}=(0,\dots,0,1,0,\dots,0)^T\in \mathbb{R}^p\), the vector whose \(k\)-th entry is \(1\) and all other entries are \(0\) for \(k=1,\dots,p\). Subsequently, we define \(\updown{e}{p}_k(k\geq 2)\), via an iterative relation analogous to that of the \(p\)-order Fibonacci sequence. 
\[\updown{e}{p}_k=\sum_{j=1}^p\updown{e}{p}_{k-j},\quad k\geq2.\]
It is easy to see that the first element of the vector \(e_k^{(p)}\) corresponds to the \(k\)-th \(p\)-order Fibonacci number \(F_k^{(p)}\).
From Lemma \ref{lemma: l_min}, we know that for  \(i\geq p\) any \(k\geq 1\), 
\[l_{i+k}\geq \sum_{j=1}^pl_{i+k-j}.\]
Denote \(L_i=(l_i,\dots,l_{i-p+1})^T\in\mathbb{R}^p\). Then we know that
\[l_{i+k}\geq \updown{e}{p}_{k+1}\cdot L_i,\quad k\geq 2-p.\]
We denote by \(\updown{e}{p,-1}_k\in\mathbb{R}^{p-1}\) the vector obtained by removing first element from \(\updown{e}{p}_k\). For instance, \(\updown{e}{p,-1}_1=\mathbf{0}\). And \(L_i^{-1}=(l_{i-1},\dots,l_{i-p+1})^T\in\mathbb{R}^{p-1}\).
When \(k=n-i\), 
we get 
\[l_i\leq (l_n-\updown{e}{p,-1}_{n-i+1}\cdot L_i^{(-1)})/\updown{F}{p}_{n-i+1},\]
which implies 
\[l_i^{\text{max}}=(1-\updown{e}{p,-1}_{n-i+1}\cdot L_i^{(-1)})/\updown{F}{p}_{n-i+1}, \ i=p,p+1,\dots,n-1.\]

When \(1\le i\leq p-1\), correspondingly, we also define associated vectors \(\updown{e}{i}_{2-k}=(0,\dots,0,1,0,\dots,0)\in \mathbb{R}^{i}\), whose 
\(k\)-th entry is \(1\) and all other entries are \(0\) for \(k=1,\dots,i\). Also denote \(\updown{e}{i}_k=(1,0,\dots,0)^T\in \mathbb{R}^i,\)  for \( k=2,\dots ,p+1-i\), since \(l_{k+i-1}\) are order statistics. Similarly, we can compute \(\updown{e}{i}_{k}(k> p+1-i)\) using the iterative relation
 \[\updown{e}{i}_k=\sum_{j=1}^p\updown{e}{i}_{k-j},\quad k> p+1-i,\]
and we obtain
\[l_{i+k}\geq \updown{e}{i}_{k+1}\cdot L_i,\quad k\geq 1,\]
 where \(L_i=(l_i,\dots,l_1)\in \mathbb{R}^i\), Then 
 \[l^{\text{max}}_i=(1-\updown{e}{i,-1}_{n-k+1}\cdot L_i^{(-1)})/m_i.\]
 for \(i=1,\dots,p-1\), where \(L_i^{(-1)}=(l_{i-1},\dots,l_1)^T\in \mathbb{R}^{i-1}\), and $m_i$ are the positive integer constants in the expressions for $l_i^{\text{max}}$ given by Lemma \ref{lemma: l_max}. When \(i=1\), \(\updown{e}{i,-1}_{n-k+1}\) is empty, which is treated as zero in computations. 

\section{Appendix 2: A simple proof}

Let \(x_1, x_2, \dots, x_n\) be independent and identically distributed random variables with sample space \((0,1)\) that represent the length of $n$ sticks, and let their probability density function be denoted by \(p(x)\). Denote by \(l_1 \le l_2 \le \dots \le l_n\leq 1\) their order statistics. The condition that these sticks cannot form a \(p+1\)-gon is that, for any \(p+1\) sticks, the sum of the lengths of the \(p\) shortest sticks is not greater than the length of the longest stick. That is \((l_1,\dots,l_n)\) should belong to the set

\[S=\{(l_1,\dots,l_n):0\leq l_1\leq \dots\leq l_n\leq 1,l_{p+i+1}\geq \sum_{j=1}^pl_{i+j},i=1,\dots,n-p-1\}.\]
Therefore, the probability that no \(p+1\) sticks can be selected from them to form a \(p+1\)-gon is given by
\[PN_n^{p+1}=n!\int_S\prod_{i=1}^n p(l_i)dl_n\dots dl_1.\]
or 
\[PN_n^{p+1}=n!\int_{l_1^{min}}^1\int_{l_2^{min}}^1\dots \int_{l_n^{min} }^1\prod_{i=1}^n p(l_i)dl_n\dots dl_1.\]
where \(l_i^{min}=l_{i-1},i=1,2,\dots,p\) and \(l_{i}^{min}=\sum_{j=1}^pl_{i-j},i=p+1,\dots,n\). Here, we consider the case of a uniform distribution, i.e., the probability density function is \(p(x) = 1\) for \(0 < x < 1\). Under this condition, the desired integral is precisely the volume of the region \(S\).

We perform the change of variables \(y_i = l_i - l_{i}^{min}\geq0\) for \(i = 1, 2, \dots, n\). Take \(\vec{y}=(y_1,\dots,y_n)^T\in\mathbb{R}^n\) and \(\vec{l}=(l_1,\dots,l_n)^T\in\mathbb{R}^n\). Under this change of variables, we obtain the transformation matrix \(J\in \mathbb{R}^{n\times n}\)(i.e., the Jacobian matrix). It can be seen that the entries of the matrix \(J\) satisfy: \(J_{ij} = 1\) when \(i = j\); for \(2 \le i \le p\) and \(j = i-1\), \(J_{i,i-1} = -1\); for \(p+1 \le i \le n\) and \(i-p \le j \le i-1\), \(J_{ij} = -1\); and all other entries are \(0\). As all diagonal entries $J_{ii} = 1$ and all entries above the diagonal are zero, the determinant of $J$ is \(1\). This means that the volume of \(S^*=JS=\{J\vec{l}:\vec{l}\in S\}\) is equal to the volume of \(S\). To compute the volume of \(S^*\), the key is to determine the upper bounds of the variables \(y_i\). Using the relation \(\vec{l}= J^{-1} \vec {y}\) together with the condition \(l_n < 1\), these upper bounds can be obtained. We know
\(l_i=y_i+l_i^{\min}\). Therefore, \(l_1=y_1\), \(l_2=y_2+y_1\). Denote by \(J_i\) the \(i\)-th row vector of \(J^{-1}\), for \(i = 1, 2, \dots, n\). Thus \(l_i=J_i\cdot y\). It can be easily deduced by induction that for \(i \le p\), \(J_i\) is the row vector whose first \(i\) components are \(1\) and the remaining components are \(0\). For the remaining \(J_i\) with \(i > p\), they satisfy 
\[J_i=\sum_{j=1}^pJ_{i-j}+e_i.\]
Where \(e_i \in \mathbb{R}^n\) denote the unit vector whose \(i\)-th component is \(1\) and all other components are \(0\). For example, \(J_{p+1}=(p,p-1,p-2,\dots,2,1,1,0,\dots,0)\). 
Denote \(J_n=(m_1,\dots,m_n)\). Then \(l_n=J_n\cdot \vec{y}\leq 1\). Thus, we have obtained a complete characterization of \(S^*\), namely
\[S^*=\{y_1\geq0,\dots,y_n\geq0, \sum_{i=1}^nm_iy_i\leq1\}.\]
Take \(z_i=m_iy_i\), its Jacobian determinant is \(\left(\prod_{i=1}^n m_i\right)\), and the integration region is defined by \(z_i \ge 0\), \(\sum z_i \le 1\), whose volume is known to be \(\frac{1}{n!}\). Therefore, we obtain the value of the original integral as \(\frac{1}{n!\prod_{i=1}^n m_i}\). Consequently, we obtain
\[PN_n^{p+1}=\prod_{i=1}^n\frac{1}{m_i}.\]
It is easy to see that the \(m_i\) here are exactly the same as those in Lemma (\ref{lemma: l_max}), and their specific values are given in Lemma (\ref{lemma: mi_values}). 
To keep the proof in the appendix self‑contained and logically consistent, we present an alternative proof of Lemma\ref{lemma: mi_values} here. 
It is easy to see that when \( i \geq p-2 \), we have \( m_i = F^p_{n+1-i} \). 
The key point lies in the case \( i < p-2 \). Before proceeding, we first prove a lemma. We need to prove the following identity.

\begin{lemma}
  For any natural number \( t  \), we have
  \[2^{t+1}-\sum_{j=1}^{t}j2^{t-j}-t-1=1.\]
\end{lemma}
\begin{proof}
When \( t = 0 \), the summation is taken as \( 0 \), and a simple calculation shows that the equality holds; when \( t \geq 1 \), For the summation part, we have \(\sum_{j=1}^tj2^{t-j}=2^t{\sum_{j=1}^t}j2^{-j}\). On the other hand, we can get
\begin{equation*}
    \begin{aligned}
        \sum_{j=1}^tjx^{j}=x\big(\sum_{j=0}^tx^j\big)' =x\big(\frac{1-x^{t+1}}{1-x}\big)'
        =\frac{x-x^{t+2}}{(1-x)^2}-\frac{(t+1)x^{t+1}}{1-x}.
    \end{aligned}
\end{equation*}
Let \(x=\frac{1}{2}\), then 
\[\sum_{j=1}^tj2^{t-j}=2^t\big(\frac{2^{t+1}-1}{2^t}-(t+1)\frac{1}{2^t}\big)=2^{t+1}-1-(t+1).\]
Substituting, we obtain the desired identity.
\end{proof}

Fix \( i \). We know that \( m_i \) is given by the \( n \)-th term of a recurrence sequence \( \{T_j\}_{j\geq 1} \), with initial conditions \( T_1 = \dots = T_{i-1} = 0 \), \( T_i = \dots = T_p = 1 \). And \(T_k=\sum_{j=1}^pT_{k-j}\) for \(k> p\). Let 
\[ T'_k=F^p_{k-i+1}-\sum_{j=1}^{p-1-i}jF^p_{k-i-j}, \]
note that the right-hand side clearly satisfies the recurrence relation for \( F^p_k \). Since the sequence is uniquely determined by its initial values and recurrence, it suffices to verify that the sequence \(T_k'\) meets the initial conditions of \( T_k \). First, for \(1 \leq k < i\), all of the subscript indices on the $F_l^p$ numbers in the expression for $T_k'$ are \(\leq 0\), and the smallest subscript value in the summation is $k-p+1$, which occurs when \(j=p-1-i\). However, because we have \(k \geq 1\), this smallest subscript value is \(\geq 2-p\) and, therefore, all of the $F_l^p$ numbers and, hence, \(T'_k \) are equal to zero. For \(k=i\), \(T'_k=F^p_1=1\);  When \( p \geq k > i \), we have the following result
\begin{equation*}
    \begin{aligned}
        T'_k&=F^p_{k-i+1}-\sum_{j=1}^{k-1-i}jF^p_{k-j-i}+\sum_{j=k-i}^{p-1-i}jF^p_{k-j-i}\\
        &=2^{k-i-1}-\sum_{j=1}^{k-2-i}j2^{k-j-i-2}-(k-i-1),
    \end{aligned}
\end{equation*}
Set \( t = k - i - 2 \). Then by the lemma, we obtain \(T'_k=1\). Thus we have \(T_k=T'_k\) for all \(k\). Setting \(k=n\), we obtain
\[m_i=T_n=F^p_{n-i+1}-\sum_{j=1}^{p-1-i}jF^p_{n-i-j}.\]

We may also discuss the case of a uniform distribution supported on \((a, 1)\), where \(0 < a < 1\). It suffices to perform the translation \(y_1 = l_1 - a\), and the rest of the procedure remains exactly the same as before. In this case, the region whose volume we need becomes \[S^{*}(a)=\{y_1\geq0,\dots,y_n\geq0,m_1(y_1+a)+\sum_{i=2}^nm_iy_i\leq 1\}.\]
The inequality \(m_1(y_1+a)+\sum_{i=2}^nm_iy_i\leq 1\) tells us that \(a < 1/m_1\); otherwise, \(S^*(a)\) is the empty set, meaning the corresponding probability is zero. 
Let \(Z(a)=\{0\leq y_1\leq a,y_2\geq 0,\dots,y_n\geq 0, \sum_{i=1}^nm_iy_i\leq 1\}\), then \[|S^*(a)|=|S^*|-|Z(a)|,\]
Here, \(|\cdot|\) denotes the volume. Let \(Z' = \{y_2\geq 0,\dots,y_n\geq 0, \sum_{i=1}^nm_iy_i\leq 1-m_1y_1\}\); then the volume of \(Z'\) can be found to be \((1-m_1y_1)^{n-1}\big((n-1)!\prod_{i=2}^nm_i\big)^{-1}\). Thus
\[|Z|=\int_0^a\frac{(1-m_1y_1)^{n-1}}{(n-1)!\prod_{i=2}^nm_i}dy_1=\frac{1-(1-m_1a)^n}{n!\prod_{i=1}^nm_i}.\]
Subtracting this from the volume of \(S^*\), we obtain that the volume of \(S^*(a)\) is \((1-m_1a)^{n}\big(n!\prod_{i=1}^nm_i\big)^{-1}\). Together with the fact that \(\prod_{i=1}^n p(l_i)\) is identically equal to \(1/(1-a)^n\), we obtain that the final probability is 
\begin{equation*}
       \frac{(1-m_1a)^n}{(1-a)^n}\cdot\frac{1}{\prod_{i=1}^n m_i}, \text{ for }  a\leq 1/m_1. 
\end{equation*}


\medskip
\noindent MSC2020: 11B39, 33C05
\end{document}